\documentclass{siamart250211}
\usepackage[utf8]{inputenc}
\usepackage{fullpage}
\usepackage{amsmath}
\usepackage{amssymb}
\usepackage{physics}
\usepackage{dsfont}
\renewcommand\normalsize{\fontsize{11}{13}\selectfont}

\newcommand{\ds}{\displaystyle}

\newtheorem{remark}{Remark}[section]

\newcommand{\bigbold}[1]{{\large \textbf{#1}}}

\title{Generalized UGK Scheme in the Diffusive Limit}

\author{
  Nicolas Crouseilles\thanks{Univ Rennes, INRIA (MINGuS), CNRS, IRMAR, UMR 6625, Rennes, France 
  (\email{nicolas.crouseilles@inria.fr}).}
  \and
  Julien Mathiaud\thanks{Univ Rennes, CNRS, IRMAR, UMR 6625, Rennes, France 
  (\email{julien.mathiaud@univ-rennes.fr}).}
  \and
  Luc Mieussens\thanks{University of Bordeaux, Bordeaux INP, CNRS, IMB, UMR 5251, Talence, France 
  (\email{luc.mieussens@math.u-bordeaux.fr}).}
}

\begin{document}

\tableofcontents
  \maketitle
  \begin{abstract}
      The unified gas kinetic scheme (UGKS) was initially designed to address multiscale challenges in rarefied gas dynamics and then extended to radiative transfert theory, as described by BGK like relaxation models. In this work, we extend its application to linear kinetic models with non isotropic scattering collision operators, as well as Fokker-Planck models . These problems typically exhibit a fully diffusive nature in the optically thick limit (corresponding to a small Knudsen number). It still leads to an asymptotic preserving (AP) property not only in this diffusive regime but also in the free transport limit. A series of numerical experiments confirm the effectiveness of the approach.

  \end{abstract}
  
  \begin{keyword}
  Transport equations, diffusion limit, asymptotic preserving schemes, BGK, Fokker-Planck, scattering.
  \end{keyword}

\section*{Introduction}

\

Kinetic models play a fundamental role in describing particle system dynamics across various domains, such as rarefied gas dynamics (RGD), neutron transport, semiconductor physics, and radiative transfer.

From a numerical perspective, solving accurately these models presents significant challenges. The computational cost is driven by the necessity to capture the smallest microscopic scales, which constrain spatial discretization and, consequently, time steps for stability. The Knudsen number, denoted by $\varepsilon$, represents the ratio of the mean free path of particles to a macroscopic length scale and characterizes the transition between collision-dominated and free-streaming regimes. As $\varepsilon$
tends to zero, a global diffusive behavior emerges at the macroscopic level. However, standard numerical schemes for moment models do not necessarily recover the correct macroscopic diffusion equation in this limit so that asymptotic-preserving (AP) numerical methods have been developed, ensuring stability and consistency across different regimes \cite{jin1991,jin1993,jin2000,klar1998,buet2002diffusion,klar2001numerical,lemou2008new,bennoune2008uniformly,carrillo2008simulation,carrillo2008numerical,gosse2011transient,lafitte2012asymptotic, mieussens2013}. 

A notable AP approach is the Unified Gas Kinetic Scheme (UGKS), originally introduced by Xu and Huang \cite{kxu2010ugks} in the context of rarefied gas dynamics. UGKS leverages a finite volume framework where numerical fluxes incorporate information from the collision operator, allowing it to bridge different regimes. Instead, UGKS makes use of the relaxation form of the BGK collision operator to derive an approximation of the interface flux by using a Duhamel representation of the solution. In other words, this representation allows for an approximation of the solution of the generalized Riemann problem at the interface that accounts for collisions.  Since its conception, UGKS has been extended to more complex gas flows \cite{liu2017unified,zhu2021first} and applied to linear models with a diffusion limit \cite{mieussens2013,sun2015radiativetransfer}. A key advantage of UGKS is that it does not require decomposition of the distribution function (like the micro-macro or the odd-even decomposition), thus avoiding issues related to boundary conditions, and it operates without staggered grids, facilitating  multi-dimensional implementations.

By construction, the extension of UGKS to non relaxation operators is not obvious. Among these collision operators, we mention the original Boltzmann operator and Fokker-Planck models of Rarefied Gas Dynamics, and non isotropic scattering and Fokker-Planck models in neutron transport and radiative transfer, for instance. A first attempt \cite{sun2015radiativetransfer} was to use the relaxation technique of Filbet-Jin (\cite{Filbet2011}): the UGKS is applied to a asymptotically consistent BGK operator, while the deviation to the original collision operator is used as a source terme. However, this was proved to induce a non physical statbility constraint on the collision kernel, at least for the diffusion regime.  Another related extension was proposed by Liu et al. for rarefied gas Dynamics (\cite{Liu2016}), where the BGK model is still used for the UGKS flux, while a the collision operator is replaced by a convex combination between the Boltzmann operator and the BGK one. Here, we propose a related approach that looks more general, in which we make use of the eigenvalue of the pseudo-inverse of the collision operator to derive a new relaxation form. Combined with the UGKS approach, and various simplifications, this allows to capture the correct diffusion limit without any stability restriction, for both non isotropic scattering collision operator and the Fokker-Planck operator as well. 
 We demonstrate that UGKS offers a robust framework for AP numerical schemes in the context of the diffusive limit. By leveraging kinetic-inspired flux approximations, UGKS enables a seamless transition between kinetic and macroscopic scales, making it a promising approach for modeling various physical phenomena. 

The rest of this work is organized as follows: first, the continuous and semi-discrete (in velocity) models are presented, and their diffusion limit is discussed. In the second section, the new UGKS framework is presented and the AP property is verified. Section 3 proposes some possible extensions and links with the UGKS scheme are discussed. Then, in Section 4, some numerical results are presented to illustrate the capabilities of the new approach.   

\section{Continuous models}
In this section, we present  models at  continuous level and discretized level: some properties and the difusion limit are given. 

\phantomsection

\subsection{ Linear collision operators with continuous velocities}

We consider the equation on the particle density function $f(t,x,v)$ which depends on time $t\geq 0$, position $x\in [0, L]\subset \mathbb{R}$ ($L>0$) and velocity $v\in[-c,c]\subset \mathbb{R}$ (with $c$ a typical velocity  of the phenomena):
\begin{equation}
\eta \partial_t f + v \partial_x f = \frac{\sigma}{\varepsilon} {\cal D} f,
\label{eq:boltzmann}
\end{equation}
$\varepsilon$ and $\eta$ being two parameters that will vary according to limit one considers. $\sigma$ is the characteristic collision frequency. The collision operator ${\cal D}$ can be  given by:
\begin{equation}
 {\cal D}f(v) := \int_{-c}^{c} k(v, v')(f(v') - f(v)) \, dv', 
\label{eq:collision_operator_scat}
\end{equation}
or
\begin{equation}
 {\cal D}f(v) = \partial_v \Big( (c^2-v^2) \partial_v f\Big). 
\label{eq:collision_operator_diff}
\end{equation}
The collisional kernels in  \eqref{eq:collision_operator_scat} and \eqref{eq:collision_operator_diff} respectively correspond to a simple scattering model (including the BGK model for $k=1$) and a Fokker-Planck-like model projected in the $v$ direction. The function $k: (v,v')\rightarrow k(v,v')$ is the probability  of passing from state $v$ to state $v'$ and is symmetric so that $\langle k(v,\cdot ),1\rangle =1, \; \forall v\in [-c, c]$, {with $\langle f,g\rangle  = \int_{-c}^c f(v) g(v) dv$ for two $v$-dependent functions $f, g$}. 

It is well known (see \cite{lemou2008new, buet2002diffusion, gosse2011transient, klar1998, jin1993}) that in the diffusion limit, ie  $\eta=\varepsilon$ tends to zero, the function $f(t, x, v)$ matches the Chapman-Enskog expansion 
\begin{equation}
\label{chapman_continuous}
f(t, x, v) = \rho(t, x)+ \frac{\varepsilon}{\sigma} {\cal D}^{+}(v) \partial_x \rho(t, x) + {\cal O}(\varepsilon^2),  
\end{equation}
where ${\cal D}^{+}$ denotes the pseudo-inverse of ${\cal D}$ and the density $\rho(t, x)=\int_{-c}^{c} f(t, x, v)dv$ satisfies the following diffusion equation:
\begin{equation}
\partial_t \rho + \partial_x (\kappa \partial_x \rho) = 0,
\label{eq:diffusion_limit}
\end{equation}
where $\kappa = \langle v,  {\cal D}^{+} v \rangle$.  The operator ${\cal D}^+$ is defined for functions with zero mean by the following condition: for any function $\phi$ such that $\langle \phi \rangle = 0$, the function $\psi =  {\cal D}^{+} \phi$ is the unique solution of $ {\cal D} \psi = \phi$ such that {$ \langle \psi, 1 \rangle = 0 $}. 

    \subsection{Continuous Problem with Discrete Velocities}
    
   \
    
    We denote $F_j(t, x) = F(t, x, v_j)$ for $j = 1, \dots, 2N$, where $v_j$ is a symmetric regular velocity grid with step size $\ds \Delta v =\frac{2c}{2N}= \frac{c}{N}$ , with $\ds v_j = -c + \frac{\Delta v}{2} + (j-1)\Delta v$ that ranges from $-c$ to $c$,  and $F(t, x) \in \mathbb{R}^{2N}$ is the solution of the following equation:
    \begin{equation}
    \label{eq:semidiscrete}
    \displaystyle
        \partial_t F_j + \frac{1}{\eta} v_j \partial_x F_j = \frac{\sigma}{\eta\varepsilon} \langle e_j, D F \rangle,
    \end{equation}
    where $D \in {\cal M}_{2N,2N}(\mathbb{R})$ is a matrix discretizing a collision operator ${\cal D}$ (diffusion \eqref{eq:collision_operator_diff} or scattering \eqref{eq:collision_operator_scat}), $e_j$ is the $j$-th vector of the canonical basis of $\mathbb{R}^{2N}$, and for $U,W\in \mathbb{R}^{2N}$ we denote $\langle U,W\rangle =\sum_{k=1}^{2N} U_k W_k$ {(let us remark the same notation as in the continuous case is used)}. Let also introduce the notation $V\in \mathbb{R}^{2N}$ which is the vector of discrete velocities 
    \begin{equation}
        \label{def:V}
        V = \begin{pmatrix} v_1 \ \hdots \ v_{2N} \end{pmatrix}^T \in \mathbb{R}^{2N},  
    \end{equation}
which satisfies  $\sum_{k=1}^{2N}v_k=0$. 

From now on we choose {$\bf c=1$} so that {$\bf \Delta v=\frac1N$}.
    
First of all, we recall some useful properties for the matrix $D$  that are common to discretization of operators like  \eqref{eq:collision_operator_diff} and \eqref{eq:collision_operator_scat} (see subsections \ref{subsecfp}  and \ref{subsecscat} for examples of such discretization matrices). We  assume the matrix $D$ to be symmetric, negative, and its kernel is
Ker $D=$ Span$(\bigbold{1})$, with 
\begin{equation}
\label{def-1}
\bigbold{1} = \begin{pmatrix} 1 \ \hdots \ 1 \end{pmatrix} \in \mathbb{R}^{2N}.  
\end{equation}
It implies that $\forall i, \;\;  \sum_{j=1}^{2N}D_{i,j}=\sum_{j=1}^{2N} D_{j,i}=0$.

\

Moreover, we introduce the following notations:
    \begin{itemize}
    \item $\lambda_k, k=0, \dots, N'$ the eigenvalues of $D$, with $N'\leq 2N-1$ and $\lambda_0=0$. 
      \item $P_0$ the orthogonal projector onto Ker $D$:   
      $P_0 = \frac{1}{2N} \begin{pmatrix} 1 & \ldots & 1 \\ \vdots & & \vdots \\ 1 & \ldots & 1 \end{pmatrix}$, associated with the eigenvalue $\lambda_0=0$, 
      \item $P_1, \dots, P_{N'}$ the orthogonal projectors onto the other eigenspaces ($N' \leq 2N-1$) associated with eigenvalues $\lambda_1, \dots, \lambda_{N'}$.
       \end{itemize}

\medskip       
\medskip       
       
 We also assume that the matrix $D$  satisfies:
 \begin{itemize}
     \item A positiveness condition: \begin{eqnarray}
D_{i,j} \geq 0 \mbox{ for } i \neq j
\end{eqnarray} .
\item  $D$ is somehow "irreducible": there exists $\delta_0>0$ such that $\forall \delta \mbox{ such that }  0\leq \delta\leq \delta_0$,  $I+\delta D$ is an irreducible bistochastic matrix: it implies that there is a sequence $k_0,.....,k_n$ of integers between $1$ and $2N$ that contains all the integers between $1$ and $2N$ such that $2$ consecutive terms are different ($k_j\ne k_{j+1}$) and $D_{k_j,k_{j+1}}> 0$
 \end{itemize}

\

Then, we recall some useful properties in the following proposition that can be proved using standard linear algebra tools for symmetric matrices. 
    \begin{proposition}    \label{prop1}
    Let consider a matrix $D\in {\cal M}_{2N, 2N}(\mathbb{R})$ symmetric, negative, and such that Ker $D=$ Span$(\bigbold{1})$ with $\bigbold{1}$ defined in \eqref{def-1}. Then, we have  
    \begin{itemize}
    \item the orthogonal projectors $P_k$ associated with the eigenvalues $\lambda_k$ of $D$ satisfy 
    $$
    \sum_{k=0}^{N'} P_k = I,\qquad \;\; 
    \ds \sum_{k=0}^{N'} {\lambda_k} P_k=\sum_{k=1}^{N'} {\lambda_k} P_k = D, \;\;\,   \mbox{ and } \;\; P_kP_j=0 \mbox{ if } k \ne j, \;\; N'\leq 2N-1.
     $$
    \item the pseudo-inverse $D^{+}$ satisfies 
    $$
    D^{+}:= \sum_{k=1}^{N'} \frac{1}{\lambda_k} P_k
   \;
     \mbox{ and } \;\;
     D^{+}D= DD^{+} = I - P_0.
$$
    \item there exists a unique $U\in \mathbb{R}^{2N}$ such that $\langle U, \bigbold{1} \rangle = 0$ and $\ds DU = V$,  being given by \eqref{def:V}. We  define  $\lambda_\star$ as:
    \begin{equation}
    \label{lambdastar}
    \ds \lambda_\star: = \frac{\langle V, V \rangle}{\langle D^{+} V, V \rangle} = \frac{\langle DU, DU \rangle}{\langle U, DU \rangle}.
    \end{equation}
    the pseudo-eigenvalue. We have $U = V / \lambda_\star$  when $DV=\lambda V$.  
    \end{itemize}
    \end{proposition}
 
    \medskip
    
    \begin{remark}
    Let us perform some remarks regarding Proposition \ref{prop1} 
    The pseudo-eigenvalue $\lambda_\star$ is always negative since $D$ is negative. For the BGK operator, we have $\lambda_\star = -1$ whereas for the Fokker-Planck operator, we have $\lambda_\star = -2$, $V$ being an eigenvector in both cases. Besides, last item leads to the second principle on the  semi-discrete equation \eqref{eq:semidiscrete}. 
    \end{remark}

    \subsection{Generic Properties of the discrete model and its Diffusive limit} In view of deriving a suitable Duhamel formula, we introduce a $\lambda_\star$-relaxation term in \eqref{eq:semidiscrete} to get the equivalent following form (let recall that $\lambda_\star$ is  negative) : in other words, we rewrite $D$ as $D=(D-\lambda_\star I) + \lambda_\star I$ so that the discrete equation is
    \begin{equation}
    \label{RTEdisc}
        \partial_t F_j + \frac{1}{\eta} v_j \partial_x F_j = \langle e_j, \frac{\sigma}{\eta\varepsilon} \left[ (D - \lambda_\star I) F + \lambda_\star F \right] \rangle. 
    \end{equation}
 Then, from \eqref{RTEdisc}, the relaxation term $\lambda_\star F$ is used to write   a Duhamel-type formula 
    \begin{eqnarray*}
        F_j(t,x) &=& \exp\left({\frac{\lambda_\star \sigma t}{\eta \varepsilon}}\right) F_j\Big({0}, x - \frac{v_j t}{\eta}\Big) \nonumber\\
        &&+ \langle e_j, \int_0^t \frac{\sigma}{\eta\varepsilon} \exp\left({\frac{\lambda_\star \sigma}{\eta\varepsilon}(t - t')}\right) (D - \lambda_\star I) F\left(t', x + \frac{v_j (t' - t)}{\eta}\right) dt' \rangle.  \nonumber    
        \end{eqnarray*}
We then perform the change of variable $u = \frac{\sigma(t - t')}{\eta\varepsilon}$ 
in the integral term to get 
\begin{eqnarray}
        F_j(t,x) &=& \exp\left({\frac{\lambda_\star \sigma t}{\eta \varepsilon}}\right) F_j\Big({0}, x - \frac{v_j t}{\eta}\Big) \nonumber\\
        &&+ \langle e_j, \int_0^{\frac{\sigma t}{\eta\varepsilon}} \exp\left({\lambda_\star u}\right) (D - \lambda_\star I) F\left(t - \frac{\eta\varepsilon u}{\sigma}, x - \frac{v_j \varepsilon u}{\sigma}\right) du \rangle. 
            \label{duhamel}
    \end{eqnarray}
From the latter expression, we can formally derive a Chapman-Enskog expansion for $F_j$ in the diffusion limit thanks to the second principle:  
\begin{proposition}[Second principle]
\label{prop12}
The following equivalences hold:
\begin{enumerate}
\item $DF=0$,
\item $F=\rho\bigbold{1}$, {with $\rho=\sum_{j=1}^{2N}F_j \Delta v$.}
\item $\langle DF,\ln(F)\rangle=0$.
\end{enumerate}
\end{proposition}
 \begin{proof}
  Using the properties of $D$, the two first assertions are equivalent. The first assertion clearly leads to the third one. It remains to prove that last assertion implies the second one. To do so, let us expand the scalar product: {$\langle DF,\ln(F)\rangle=\langle F,D\ln(F)\rangle$}
  \begin{eqnarray*}
   \langle DF,\ln(F)\rangle&=& \sum_{1\leq i,j\leq 2N} F_i D_{i,j}\ln(F_j) 
  \end{eqnarray*}
  Since $D\bigbold{1}=0$   and  $D$ is symmetric, one gets that $\forall i, D_{i,i}=-\sum_{j\neq i}D_{i,j}=-\sum_{j\neq i}D_{j,i}$ so that 
   \begin{eqnarray*}
   \langle DF,\ln(F)\rangle&=& \sum_{1\leq i\neq j\leq2N} F_i D_{i,j}\ln(F_j) -\sum_{1\leq i\neq j\leq 2N} F_i D_{i,j}\ln(F_i)\\
   &=& \sum_{1\leq i\neq j\leq2N} F_i D_{i,j}\ln(F_j/F_i)\\
     &=& {\frac12 \sum_{1\leq i\neq j\leq2N} F_i D_{i,j}\ln(F_j/F_i) + \frac12 \sum_{1\leq i\neq j\leq2N} F_j D_{i,j}\ln(F_i/F_j)}\\
   &=& \frac12\sum_{1\leq i\neq j\leq2N} (F_i-F_j) D_{i,j}\ln(F_j/F_i)
  \end{eqnarray*}
  If the scalar product $ \langle DF,\ln(F)\rangle$ is equal to zero, 
  we get that as soon as $D_{i,j}>0$ $F_i=F_j$ since $(x,y)\rightarrow (y-x)\ln(x/y)$ is a negative function which is equal to zero only when $x=y$.
  
  Since there is a sequel $k_0,.....,k_n$ of integers between $1$ and $2N$ that contains all the integers between $1$ and $2N$ such that $2$ consecutive terms are different ($k_j\ne k_{j+1}$) and $D_{k_j,k_{j+1}}> 0$, all $F_i$ are equal leading to the fact that $F=\rho\bigbold{1}$.
 \end{proof}   
    \subsection{Diffusion Limit}
    
   We now  study the diffusion limit $\varepsilon=\eta\to0$ using Duhamel formula \eqref{duhamel}.  
    
    \phantomsection
    
    \subsubsection{Zero-th Order }
    
    Passing to the limit $\varepsilon = \eta \to 0$ in \eqref{duhamel}, we obtain from the dominated convergence theorem (recalling that $\lambda_\star < 0$) 
    \begin{eqnarray}
        F_j(t,x) &\to& \langle e_j, \left( D - \lambda_\star I \right) F(t,x) \int_0^{\infty} \exp(\lambda_\star u) {du} \rangle \nonumber\\
        &&= -\frac{1}{\lambda_\star} \langle e_j, \left( D - \lambda_\star I \right) F(t,x) \rangle. 
    \end{eqnarray}
Hence,   at the zero-th order, the last relation rewrites as $F(t,x) = \left( I - \frac{1}{\lambda_\star} D \right) F(t,x)$, so that
    \begin{equation}
        D F(t,x)=0. \label{inverse}
    \end{equation}
    From the properties of $D$ (see Proposition \ref{prop12}), we deduce that{$F(t,x)=\rho(t,x) \bigbold{1}$ with $\rho(t,x) = \sum_{j=1}^{2N} F_j(t, x) \Delta v$.} 
    
    \subsubsection{First order}
    We now go to the next order 
    by performing a Taylor expansion of $F$ in \eqref{duhamel}:
    \begin{eqnarray*}
        F\left(t - \frac{\eta\varepsilon u}{\sigma}, x - \frac{v_j \varepsilon u}{\sigma}\right) &=& F(t,x) - \frac{\eta\varepsilon u}{\sigma} \frac{v_j}{\eta} \partial_x F(t,x) + O(\eta \varepsilon) \nonumber \\
        &=& F(t,x) - \frac{\varepsilon u v_j}{\sigma} \partial_x \rho(t,x) \bigbold{1} + O(\eta \varepsilon) \nonumber,
    \end{eqnarray*}
    where we used the zero-th order approximation of $F$ to set $\partial_x F(t,x) = \partial_x \rho(t,x) \bigbold{1} + {\cal O}(\varepsilon)$. Thus, the term in the integral in \eqref{duhamel} becomes (recalling that $D \bigbold{1} = 0$):
    \begin{align*}
        &\langle e_j, (D - \lambda_\star I) F\left(t - \frac{\eta\varepsilon u}{\sigma}, x - \frac{v_j \varepsilon u}{\sigma}\right) \rangle\\
        =& \langle e_j, (D - \lambda_\star I) F(t,x) \rangle - \frac{\varepsilon u}{\sigma} \partial_x \rho(t,x) \langle e_j, (D - \lambda_\star I) v_j \bigbold{1} \rangle + O(\varepsilon^2) \nonumber \\
        =& \langle e_j, (D - \lambda_\star I) F(t,x) \rangle + \frac{\varepsilon u}{\sigma} \partial_x \rho(t,x) \lambda_\star v_j + O(\varepsilon^2). \nonumber
    \end{align*}
    Multiplying the latter expression by $\exp(\lambda_\star u)$ and integrating over $u \in [0, +\infty[$,  equation \eqref{duhamel} becomes:
    \begin{multline}
        \int_0^\infty \exp(\lambda_\star u) (D - \lambda_\star I) F\left(t - \frac{\eta\varepsilon u}{\sigma}, x - \frac{v_j \varepsilon u}{\sigma}\right) du \notag=\\ \left( I - \frac{1}{\lambda_\star} D \right) F(t,x) + \frac{\varepsilon}{\lambda_\star \sigma} \partial_x \rho(t,x) V + O(\varepsilon^2),
    \end{multline}
    where we used $\int_0^{+\infty} u \exp(\lambda_\star u) du = 1/\lambda_\star^2$ and $\int_0^{+\infty} \exp(\lambda_\star u) du = -1/\lambda_\star$. Consequently, 
    since the first term in \eqref{duhamel} tends to zero exponentially fast, \eqref{duhamel} becomes, as $\varepsilon\to 0$
    \begin{eqnarray*}
        F(t,x) = \left( I - \frac{1}{\lambda_\star} D \right) F(t,x) + \frac{\varepsilon}{\lambda_\star \sigma} \partial_x \rho(t,x) V + O(\varepsilon^2), 
    \end{eqnarray*} 
    or after simplifying $F(t, x)$ on both sides and multiplying by $\lambda_\star$  
    \begin{equation}
    \label{CE_F}
       D  F(t,x) = \frac{\varepsilon}{\sigma} \partial_x \rho(t,x) V + O(\varepsilon^2).  
    \end{equation}
   Applying the pseudo-inverse $D^+$ (whose properties are recalled in Proposition \ref{prop1}) leads to 
    \begin{eqnarray}
        F(t,x) = \rho(t,x) {\bigbold{1}} + \frac{\varepsilon}{\sigma} \partial_x \rho(t,x) D^{+} V + O(\varepsilon^2),
    \end{eqnarray}
    which is indeed the discrete counterpart of the desired result \eqref{chapman_continuous}.
    
    \begin{remark}
    In \eqref{duhamel}, The Duhamel formula has been written on the time interval $[0, t]$ but it can 
also be written on the time interval $[t_n, t]$. In this case, it comes 
    \begin{align}  \label{duhamel2}
        F_j(t, x) 
        &= \exp\left(\frac{\lambda_\star \sigma (t - t_n)}{\eta \varepsilon}\right) F_j\Big(t_n, x - \frac{v_j (t - t_n)}{\eta}\Big) \nonumber\\
        &+ \langle e_j, \int_{0}^{\frac{\sigma (t - t_n)}{\eta\varepsilon}} \exp( \lambda_\star u) (D - \lambda_\star I) F\left(t - \frac{\eta\varepsilon u}{\sigma}, x - \frac{v_j \varepsilon u}{\sigma}\right) du \rangle. 
    \end{align}
    \end{remark}
    
    \section{Numerical scheme}
    In this section, we will describe 
    a generalized UGKS for linear equations  of the form \eqref{eq:semidiscrete}. 
First of all, we introduce a uniform spatial mesh $x_i = i \Delta x$ is introduced and denote $x_{i+1/2} = (x_i + x_{i+1})/2$ the interface between two cells. The spatial interval being $x\in [0, L], L>0$, the mesh step is defined as $\Delta x=L/N_x$, $N_x$ being the number of cells. Moreover, we introduce the time discretization 
    $t_n=n\Delta t, \Delta t>0, n\in \mathbb{N}$. 

    To derive a UGK scheme, one of the main ingredient relies on a suitable interface value $F_j(t, x_{i+1/2})$, usually based on a Duhamel formula (see \cite{kxu2010ugks, mieussens2013}), that will serve in the finite volume formulation as a  flux approximation of the space derivative in \eqref{eq:semidiscrete}. Another ingredient is the space reconstruction of $F_j(t, x)$ in the  integral term of Duhamel formula \eqref{duhamel2}.     
    After recalling the UGKS, we will present some spatial reconstructions and we will see that the general context considered here induces some difficulties.  
    
    \phantomsection
    
    \subsection{UGKS}
    Our goal is to design a UGKS type numerical scheme for a general class of linear equations \eqref{eq:boltzmann}. Let recall the basics of UGKS framework (see \cite{kxu2010ugks, mieussens2013}). First, starting from  \eqref{eq:semidiscrete}, we define the averages of the density and distribution function on cell $i$ at time $t_n$ 
    $$
\left(    
\begin{matrix}
\rho^n_i \\    
F^n_{i,j}
\end{matrix}
\right) 
    = 
 \frac{1}{\Delta x}  \int_{x_{i-1/2}}^{x_{i+1/2}} \left(    
\begin{matrix}
\rho(t_n, x) \\    
F_j(t_n, x)
\end{matrix}
\right)dx,  
$$
    and the macroscopic and microscopic numerical fluxes across the interface $x_{i+1/2}$ 
$$
\left(    
\begin{matrix}
\Phi_{i+1/2} \\    
\phi_{i+1/2,j}
\end{matrix}
\right) 
    = 
\frac{1}{\eta\Delta t}  \int_{t_n}^{t_{n+1}} 
\left( \begin{matrix}
\sum_{j=1}^{2N} v_j F_j(t, x_{i+1/2})\Delta v  \\    
v_j F_j(t, x_{i+1/2})
\end{matrix}
\right)dt. 
$$
The finite volume formulations of the  macroscopic and  kinetic  equations are thus 
\begin{eqnarray}
\label{ugks_rho}
\frac{\rho^{n+1}_i-\rho^n_i}{\Delta t} +\frac{1}{\Delta x}(\Phi_{i+1/2}-\Phi_{i-1/2}) &=& 0, \\
\label{ugks_f}
\frac{F^{n+1}_{i,j}-F^{n}_{i,j}}{\Delta t} +\frac{1}{\Delta x}(\phi_{i+1/2,j}-\phi_{i-1/2,j}) &=& \frac{\sigma}{\eta\varepsilon} DF^{n+1}_{i,j}, 
\end{eqnarray}
where the collision term is implicit for stability reason. 

As we can see, the interface value of 
$F_j(t, x_{i+1/2})$ plays an important role since it enables to compute the numerical fluxes of the finite volume method. 
As mentioned above, the main idea of UGKS relies on the suitable approximation of the interface value of 
$F_j(t, x_{i+1/2})$ obtained from a space approximation of a Duhamel formula. Indeed, the interface value is obtained from a suitable space approximation of the Duhamel formula of the original equation.  Here, we will use the Duhamel formula \eqref{duhamel2} which includes a penalization procedure with the factor $\lambda_\star$ and we will discuss the space approximation that ensures consistency and good asymptotic behavior. Contrary to the BGK case studied in \cite{kxu2010ugks, mieussens2013} where the distribution function $F_j(t, x)$ and the density $\rho(t, x)$ have to be reconstructed, 
the general case considered here only involves $F_j(t, x)$. In the sequel, different reconstructions are discussed and we will see that some choices require the full knowledge of the spectral decomposition of the linear collision operator, which is not acceptable from a computational point of view. Indeed, our goal is to design a UGKS scheme which is computationally efficient, preserves the diffusion limit $\eta=\varepsilon\to 0$ 
and the collisionless limit $\eta$ fixed and $\varepsilon\to \infty$.

    \subsection{Spatial Approximation: first attempt}
    A first (natural) reconstruction of $F_j$ on the spatial mesh is:
    \begin{equation}
    \label{reconstruction1}
    F_j(t, x) = F_{i+1/2,j}(t) + \frac{F_j{}(t,x_{i+1}) - F_{j}(t,x_{i})}{\Delta x} \left(x - x_{i+1/2}\right), \;\;\; x\in [x_i, x_{i+1}].  
    \end{equation}
    The reconstruction is inserted in \eqref{duhamel} to get an approximation of the interface value $F_{i+1/2, j}(t) \approx F_j(t, x_{i+1/2})$. Using the reconstruction \eqref{reconstruction1}, the integral term in \eqref{duhamel} evaluated at $x=x_{i+1/2}$ enables to get the following approximation 
    \begin{align}
    F\left(t - \frac{\eta\varepsilon (u - t_n)}{\sigma}, x_{i+1/2} - \frac{v_j \varepsilon u}{\sigma}\right) &\approx F\left(t, x_{i+1/2} - \frac{v_j \varepsilon u}{\sigma}\right) \nonumber \\
    &\hspace{-4cm}= F_{i+1/2}(t) + \frac{F_{i+1}(t) - F_{i}(t)}{\Delta x} \left(x_{i+1/2} - \frac{\varepsilon v_j u}{\sigma} - x_{i+1/2}\right) \nonumber \\
    \label{approx_int}
    &\hspace{-4cm}= F_{i+1/2}(t) - \frac{v_j \varepsilon u}{\sigma} \frac{F_{i+1}(t) - F_{i}(t)}{\Delta x}.
    \end{align}
    Let us remark that in the first approximation the shift in time has been neglected, as usual in the UGKS for the diffusion regime. 
    
    Regarding now the first term in \eqref{duhamel}, 
    we consider a first order approximation in space based on the sign of $v_j$. We thus obtain 
    \begin{equation}
    \label{approx_cl}
    F_j\Big(t_n, x_{i+1/2} - \frac{v_j (t - t_n)}{\eta\varepsilon}\Big) \approx F_{i,j}^n \mathds{1}_{v_j > 0} + F_{i+1,j}^n \mathds{1}_{v_j < 0}, 
    \end{equation}
    where $\mathds{1}_{v_j < 0}$ denotes the Heaviside function which is equal to one if $v_j<0$ and zero else (same for $\mathds{1}_{v_j > 0}$).  
    
    Thus, evaluating \eqref{duhamel2} at $x = x_{i+1/2}$ and using the previous approximations \eqref{approx_int} and \eqref{approx_cl}, we obtain 
    the following approximation $F_{i+1/2,j}(t)$ of the interface value $F_j(t, x_{i+1/2})$:
    \begin{eqnarray}
       F_{i+1/2,j}(t) &:=& \exp\left({\frac{\lambda_\star \sigma (t - t_n)}{\eta \varepsilon}}\right) \left(F_{i,j}^n \mathds{1}_{v_j > 0} + F_{i+1,j}^n \mathds{1}_{v_j < 0}\right) \nonumber\\
         \label{discrete}
        &&\hspace{-1.2cm} + \langle e_j, \int_{0}^{{\frac{\sigma (t - t_n)}{\eta\varepsilon}}} \exp\left(\lambda_\star {u}\right) (D - \lambda_\star I) \left[F_{i+1/2}(t) - \frac{\varepsilon v_j u}{\sigma} \frac{F_{i+1}(t) - F_{i}(t)}{\Delta x}\right] du \rangle.
    \end{eqnarray}
This approximation can be inserted in \eqref{ugks_f} to get a first version of the scheme but we can observe the resulting scheme will couple the space and velocity indices so that the calculation of $F_{i,j}^{n+1}$ will be very costly.  Moreover, 
as observed in \eqref{chapman_continuous}, capturing the correct asymptotic behavior requires to introduce $D^+$ which is not the case with this first attempt. Thus, we will next try to work only with $F_{i+1/2}(t)$ and not with $F_i(t)$ and $F_{i+1}(t)$. 
    
    \subsection{Spatial Approximation: second attempt}

    Thus, instead of the reconstruction \eqref{reconstruction1}, we propose the following reconstruction:
    \begin{equation}
    \label{reconstruction}
    F_j(t, x) = F_{i+1/2,j}(t) + \frac{F_{i+1,j}^n - F_{i,j}^n}{\Delta x} \left(x - x_{i+1/2}\right),
    \end{equation}
    where we neglect the temporal variations of spatial gradients ($t_n \leq t \leq \Delta t$). 
    Thus, evaluating \eqref{duhamel2} at $x=x_{i+1/2}$, using the reconstruction \eqref{reconstruction} for the integral term and \eqref{approx_cl} for the first term, we obtain:
    \begin{eqnarray}
        F_{i+1/2,j}(t) &:=& \exp\left({\frac{\lambda_\star \sigma (t - t_n)}{\eta \varepsilon}}\right) \left(F_{i,j}^n \mathds{1}_{v_j > 0} + F_{i+1,j}^n \mathds{1}_{v_j < 0}\right) \notag \\
        \label{discrete2}
        &&\hspace{-1.2cm}+ \langle e_j, \int_{0}^{{\frac{\sigma (t - t_n)}{\eta\varepsilon}}} \exp\left({\lambda_\star u}\right) (D - \lambda_\star I) \left[F_{i+1/2}(t) - \frac{\varepsilon v_j u}{\sigma} \frac{F_{i+1}^n - F_i^n}{\Delta x}\right] du \rangle. 
    \end{eqnarray}
From the relation \eqref{discrete2}, $F_{i+1/2}(t)$ can now be expressed as the solution of the following linear system (which is local in space): 
    \begin{eqnarray}
    \label{MF=S}
    M(t) F_{i+1/2}(t) = S(t),
    \end{eqnarray}
    where the $j$-th component of the source term $S(t)$ is:
    \begin{eqnarray}
        S_j(t) &=& \exp\left({\frac{\lambda_\star \sigma (t - t_n)}{\eta \varepsilon}}\right) \left(F_{i,j}^n \mathds{1}_{v_j > 0} + F_{i+1,j}^n \mathds{1}_{v_j < 0}\right) \notag \\
        \label{defS}
        &+& \langle v_j e_j, \int_{0}^{{\frac{\sigma (t - t_n)}{\eta\varepsilon}}} \exp\left({\lambda_\star u}\right) (D - \lambda_\star I) \left[-\frac{\varepsilon {u}}{\sigma} \frac{F_{i+1}^n - F_i^n}{\Delta x}\right] du \rangle,
    \end{eqnarray}
    and the matrix $M(t)$ is given by 
    \begin{eqnarray}
    M(t) &=& I - \int_{0}^{{\frac{\sigma (t - t_n)}{\eta\varepsilon}}} \exp\left({\lambda_\star u}\right) (D - \lambda_\star I) du\nonumber\\
    \label{matrixM}
    &=&\exp\left({\frac{\lambda_\star \sigma (t - t_n)}{\eta \varepsilon}}\right) I + \left(1 - \exp\left({\frac{\lambda_\star \sigma (t - t_n)}{\eta \varepsilon}}\right)\right) \left(\frac{1}{\lambda_\star} D\right).
    \end{eqnarray} 
Some important properties of the matrix $M(t)$ are given in the following proposition.     
    \begin{proposition}
    The matrix $M(t)$ defined by \eqref{matrixM} is positive definite and thus invertible and its inverse is given by 
       \begin{eqnarray}
    \label{M-1}
    M(t)^{-1} &=& 
\sum_{k=0}^{N'} \mathcal{A}_k^{-1} P_k,      \end{eqnarray}
    where $N'$ is the number of eigenvalues, $P_k$ denotes the orthogonal projectors associated to the eigenvalues $\lambda_k$ of $D$, and ${\cal A}_k\in \mathbb{R}$ are defined as follows 
    \begin{equation}
     \label{defAk}
 {\cal A}_k = \exp\left({\frac{\lambda_\star \sigma (t - t_n)}{\eta \varepsilon}}\right) + \left(\frac{\lambda_k}{\lambda_\star}\right) \left(1 - \exp\left({\frac{\lambda_\star \sigma (t - t_n)}{\eta \varepsilon}}\right)\right). 
 \end{equation}
    \end{proposition}
    \begin{proof} 
    One can observe that $M(t)$ is a convex combination of the two positive matrices $I$ and $\frac{1}{\lambda_\star} D$ (since $\lambda_\star<0$ and $D$ is negative). 
    Moreover, it is always positive definite since $I$ is, $\left(\frac{1}{\lambda_\star} D\right)$ is positive, and the coefficient in front of $I$ is always strictly positive, so the system will always be numerically invertible. 
    
    Concerning the inverse $M(t)^{-1}$, we recall the projector properties 
    from Prop. \ref{prop1}: $\sum_{k=0}^{N'} P_k=I$ and $\sum_{k=1}^{N'} \lambda_kP_k=D$.  Indeed, using these relations in \eqref{matrixM}, one has  
    \begin{eqnarray*}
    M(t) &=& \exp\left({\frac{\lambda_\star \sigma (t - t_n)}{\eta \varepsilon}}\right) \sum_{k=0}^{N'} P_k + \left(1 - \exp\left({\frac{\lambda_\star \sigma (t - t_n)}{\eta \varepsilon}}\right)\right) \sum_{k=1}^{N'} \left(\frac{\lambda_k}{\lambda_\star}\right) P_k, \nonumber
    \end{eqnarray*}
which reads: $M(t)=\sum_{k=0}^{N'} {\cal A}_k P_k$. 
The expression \eqref{M-1}  of the inverse of $M(t)$ is deduced from this latter form.  
\end{proof}
 Let us remark that the case $k=0$ in \eqref{M-1} will play an important role since $\lambda_0=0$ and for $k=0$, we have ${\cal A}_0=\exp\left({\frac{\lambda_\star \sigma (t - t_n)}{\eta \varepsilon}}\right)$.

    \subsubsection{Computation of the Interface Value}
    From the relation \eqref{MF=S} satisfied by the interface value $F_{i+1/2}(t)$ and the expressions \eqref{defS} of $S(t)$ and \eqref{M-1} of $M(t)^{-1}$, we will compute explicitly the interface value $F_{i+1/2}(t)$. In order to avoid the explicit calculation of the projectors $P_k$ and thus to get a simple numerical scheme, some approximations will be performed. 
    
    First, let introduce some useful notations.    
        The half densities associated with positive and negative velocities are 
        \begin{equation}
            \label{half_densities}
            \rho_i^{-,n} = \frac{1}{2N} \sum_{j=1}^N F_{i,j}^n= \Delta v \sum_{j=1}^N F_{i,j}^n  \;\; \mbox{ and } \;\; \rho_i^{+, n} = \frac{1}{2N} \sum_{j=N+1}^{2N} F_{i,j}^n.   
    \end{equation}    
Then, we introduce the coefficient 
\begin{eqnarray}
       \ds \mathcal{C}(t) &=& \lambda_\star \int_{0}^{\frac{\sigma (t - t_n)}{\eta\varepsilon}} \exp\left({\lambda_\star u}\right) u du 
       \label{defC}
       = \frac{1}{\lambda_\star} \left[1 + \left({\frac{\lambda_\star \sigma (t - t_n)}{\eta \varepsilon}}-1\right) \exp\left({\frac{\lambda_\star \sigma (t - t_n)}{\eta \varepsilon}}\right)\right].
\end{eqnarray} 
We now compute the interface value $F_{i+1/2}$ using the expressions \eqref{defS} of $S(t)$ and \eqref{M-1} of $M(t)^{-1}$:      \begin{eqnarray}
    F_{i+1/2}(t)&=& M(t)^{-1} S(t) \notag \\
    &=& P_0 \left\{F_{i,j}^n \mathds{1}_{v_j > 0} + F_{i+1,j}^n \mathds{1}_{v_j < 0}\right\}_{j=1, \dots, 2N} \notag \\
    &-& \frac{\varepsilon \mathcal{C}(t)}{\sigma} \frac{1}{2N} \langle (D V - \lambda_\star V), \frac{F_{i+1}^n - F_i^n}{\Delta x} \rangle \bigbold{1} \notag \\
    &+& \sum_{k=1}^{N'} {\cal A}_k^{-1} P_k\left\{ \exp\left({\frac{\lambda_\star \sigma (t - t_n)}{\eta \varepsilon}}\right) \left(F_{i,j}^n \mathds{1}_{v_j > 0} + F_{i+1,j}^n \mathds{1}_{v_j < 0}\right)\right\}_{j=1, \dots, 2N} \notag \\
    \label{resFi}
    &+& \sum_{k=1}^{N'}{\cal A}_k^{-1} P_k \left\{ \langle v_j e_j, \frac{\mathcal{C}(t)}{\lambda_\star} (D - \lambda_\star I) \left[-\frac{\varepsilon}{\sigma} \frac{F_{i+1}^n - F_i^n}{\Delta x}\right] \rangle \right\}_{j=1, \dots, 2N},  \nonumber\\
    &=& \mbox{\mbox{\textcircled{a}}}+ \mbox{\textcircled{b}}+ \mbox{\textcircled{c}}+ \mbox{\textcircled{d}}, 
    \end{eqnarray}
    where we used the notations $\left\{ w_j \right\}_{j=1, \dots, 2N}$ the vector $w\in\mathbb{R}^{2N}$ and we remind ${\cal A}_k$ is defined by  \eqref{defAk}. Let us now detail how we deal with  terms $\mbox{\textcircled{a}},\mbox{\textcircled{b}}, \mbox{\textcircled{c}}, \mbox{\textcircled{d}}$. 
    
    \subsubsection{Computation of the terms \textcircled{a} and  \textcircled{b}} Let consider in this part the first two terms $\mbox{\textcircled{a}}$ and $\mbox{\textcircled{b}}$. First, from the definition \eqref{half_densities} of the half densities, one has for $\mbox{\textcircled{a}}$ 
    \begin{equation}
        \label{arond}
    \mbox{\textcircled{a}}=\left(\rho_i^{+,n} + \rho_{i+1}^{-,n}\right) \bigbold{1}. 
        \end{equation}
    Second, we provide some details regarding $\mbox{\textcircled{b}}$. Recalling $(D - \lambda_\star I)$ is symmetric and from the definition \eqref{defC} of ${\cal C}$, one has   
    \begin{eqnarray}
    \mbox{\textcircled{b}}&=& -\lambda_\star \frac{\varepsilon}{\sigma} \frac{1}{2N} \langle (D  - \lambda_\star I)V, \frac{{\cal C}(t)}{\lambda_\star} \frac{F_{i+1}^n - F_i^n}{\Delta x} \rangle {\bigbold{1}}\nonumber\\
    &=& \lambda_\star\frac{1}{2N} \langle (D  - \lambda_\star I)V, \int_{0}^{{\frac{\sigma (t - t_n)}{\eta\varepsilon}}} \exp\left({\lambda_\star u}\right) \left[-\frac{\varepsilon u}{\sigma} \frac{F_{i+1}^n - F_i^n}{\Delta x}\right] du \rangle {\bigbold{1}} \nonumber\\
     &=& \lambda_\star\frac{1}{2N} \langle V, (D - \lambda_\star I) \int_{0}^{{\frac{\sigma (t - t_n)}{\eta\varepsilon}}} \exp\left({\lambda_\star u}\right) \left[-\frac{\varepsilon u}{\sigma} \frac{F_{i+1}^n - F_i^n}{\Delta x}\right] du \rangle{\bigbold{1}}\nonumber\\
     \label{termb_1}
        &=& \lambda_\star\frac{1}{2N} \sum_{k=1}^{2N} \langle v_k e_k, \int_{0}^{{\frac{\sigma (t - t_n)}{\eta\varepsilon}}} \exp\left({\lambda_\star u}\right) (D - \lambda_\star I) \left[-\frac{\varepsilon u}{\sigma} \frac{F_{i+1}^n - F_i^n}{\Delta x}\right] du \rangle{\bigbold{1}}. 
    \end{eqnarray}

   \noindent Let us remark that this term is independent of $j$  so that it has no contribution in the diffusion limit.  
    
  
     \subsubsection{Computation of the terms  {\textcircled{c}} and \textcircled{d}} 
     Now, we will consider the terms $\mbox{\textcircled{c}}$ and $\mbox{\textcircled{d}}$ for which some approximations will be performed to avoid the explicit calculation of the projectors $P_k$. 

      First let recall the expression of the term $\mbox{\textcircled{c}}$
      $$
      \mbox{\textcircled{c}}=\sum_{k=1}^{N'} {\cal A}_k^{-1} P_k \exp\left({\frac{\lambda_\star \sigma (t - t_n)}{\eta \varepsilon}}\right) \left\{F_{i,j}^n \mathds{1}_{v_j > 0} + F_{i+1,j}^n \mathds{1}_{v_j < 0}\right\}_{j=1, \dots, 2N}. 
      $$
On the one side, we observe that the term $\mbox{\textcircled{c}}$ decays exponentially fast and thus does not contribute in the diffusion limit. 
On the other side, in the transport limit $\varepsilon\to +\infty$, one has 
    \begin{align}
 \sum_{k=1}^{N'} {\cal A}_k^{-1} P_k &=   \sum_{k=1}^{N'} \frac{1}{\exp\left({\frac{\lambda_\star \sigma (t - t_n)}{\eta \varepsilon}}\right) + \frac{\lambda_k}{\lambda_\star} \left(1 - \exp\left({\frac{\lambda_\star \sigma (t - t_n)}{\eta \varepsilon}}\right)\right)} P_k \nonumber\\
 &= \sum_{k=1}^{N'} P_k + O(1/\varepsilon) = (I - P_0) + O(1/\varepsilon).
    \end{align}
Hence, from these asymptotic behaviors,  and in view of constructing a method which does not require the knowledge of $P_k$, 
we propose the following approximation for $\mbox{\textcircled{c}}$  
    \begin{equation}
    \label{crond}
    \mbox{\textcircled{c}} 
    \approx  
    (I - P_0) \exp\left({\frac{\lambda_\star \sigma (t - t_n)}{\eta \varepsilon}}\right) \left\{F_{i,j}^n \mathds{1}_{v_j > 0} + F_{i+1,j}^n \mathds{1}_{v_j < 0}\right\}_{j=1, \dots, 2N}. 
    \end{equation}
    
Let now consider the term $\mbox{\textcircled{d}}$ and 
first, let recall its expression 
 \begin{equation}
     \label{def:termd}
 \mbox{\textcircled{d}}= \sum_{k=1}^{N'}{\cal A}_k^{-1} P_k \left\{ \langle v_j e_j, \frac{\mathcal{C}(t)}{\lambda_\star} (D - \lambda_\star I) \left[-\frac{\varepsilon}{\sigma} \frac{F_{i+1}^n - F_i^n}{\Delta x}\right] \rangle \right\}_{j=1, \dots, 2N}. 
 \end{equation}
Then, we observe that in the diffusion limit, we have, for any arbitrary $q\in\mathbb{N}$:
    \begin{align}
\sum_{k=1}^{N'} \frac{{\cal A}_k^{-1}}{\lambda_\star} P_k 
\label{approx4}
&= \sum_{k=1}^{N'} \frac{1}{\lambda_k} P_k + O(\varepsilon^q) = D^{+} + O(\varepsilon^q), 
    \end{align}
   since from the definition \eqref{defAk} of ${\cal A}_k$,  
   we deduce ${\cal A}_k^{-1}$ tends to $\lambda_\star/\lambda_k$ as $\varepsilon\to 0$ and from Prop. \ref{prop1}, we have $D = \sum_{k=1}^{N'} \lambda_k P_k$. However, little can be said about the term in braces in \eqref{def:termd} except that $\displaystyle 
    \lim_{\eta=\varepsilon \to 0} \mathcal{C}(t) = 1/\lambda_\star$. Moreover, using the approximation \eqref{approx4} imposes the term into braces to be orthogonal to ${\bf 1}$ which is not the case due to the presence of $F^n_{i+1}-F_i^n$. Hence, the calculation of the term $\mbox{\textcircled{d}}$ would require the knowledge of the projectors $P_k$ which we want to avoid since it may be very costly in the general case. 
    Hence, in the next subsection, we will consider another reconstruction which will avoid the knowledge of the projectors $P_k$. 
    
    \subsection{Spatial Approximation: third attempt}
    Due to the obstacle observed previously, we thus consider the following reconstruction for $F_j(t, x)$  
    \begin{equation}
        \label{third}
         F_j(t, x) = F_{i+1/2,j}(t) + \frac{\rho_{i+1}^n-\rho_{i}^n}{\Delta x}(x-x_{i+1/2}). 
    \end{equation}
    As before, assessing \eqref{duhamel2} at $x=x_{i+1/2}$, using  reconstruction \eqref{third} for the integral term and \eqref{approx_cl} for the first term, we get 
the following interface relation for $F_{i+1/2, j}(t)$ 
    \begin{eqnarray}
    F_{i+1/2, j}(t)&=& \exp\left({\frac{\lambda_\star \sigma (t - t_n)}{\eta \varepsilon}}\right) \left(F_{i,j}^n \mathds{1}_{v_j > 0} + F_{i+1,j}^n \mathds{1}_{v_j < 0}\right) \notag \\
        \label{discrete3}
        &&\hspace{-1.2cm}+ \langle e_j, \int_{0}^{{\frac{\sigma (t - t_n)}{\eta\varepsilon}}} \exp\left({\lambda_\star u}\right) (D - \lambda_\star I) \left[F_{i+1/2}(t) - \frac{\varepsilon v_j {{u}}}{\sigma} \frac{\rho_{i+1}^n - \rho_i^n}{\Delta x}{\bf 1}\right] du \rangle. 
    \end{eqnarray}
    The same calculations as the ones done before can be performed since the expression of $M(t)^{-1}$ given by \eqref{M-1} is unchanged, but the expression of $S(t)$ in \eqref{defS} is slightly modified since its expression is now 
    \begin{eqnarray}
        S_j(t) &=& \exp\left({\frac{\lambda_\star \sigma (t - t_n)}{\eta \varepsilon}}\right) \left(F_{i,j}^n \mathds{1}_{v_j > 0} + F_{i+1,j}^n \mathds{1}_{v_j < 0}\right) \notag \\
        \label{defS2}
        &+& \langle v_j e_j, \int_{0}^{{\frac{\sigma (t - t_n)}{\eta\varepsilon}}} \exp\left({\lambda_\star u}\right) (D - \lambda_\star I) \left[-\frac{\varepsilon {u}}{\sigma} \frac{\rho_{i+1}^n - \rho_i^n}{\Delta x}{\bf 1}\right] du \rangle. 
    \end{eqnarray}
    So now $F_{i+1/2}$ satisfies $F_{i+1/2}(t) = M(t)^{-1}S(t)$ and decompose
    $F_{i+1/2}(t) = \mbox{\textcircled{a}}+\mbox{\textcircled{b}}+\mbox{\textcircled{c}}+\mbox{\textcircled{d}}$ as: 
    \begin{eqnarray}
    \mbox{\textcircled{a}}&=&P_0\left\{F^n_{i,j}\mathds{1}_{v_j > 0} + F_{i+1,j}^n \mathds{1}_{v_j < 0}\right\}_{j=1, \dots, 2N}\nonumber\\
    \mbox{\textcircled{b}}&=&-\frac{\varepsilon \mathcal{C}(t)}{\sigma} \frac{1}{2N} \langle (D V - \lambda_\star V), \frac{\rho_{i+1}^n - \rho_i^n}{\Delta x} \rangle \bigbold{1}\nonumber\\
    \mbox{\textcircled{c}}&=&\sum_{k=1}^{N'} {\cal A}_k^{-1} P_k \exp\left({\frac{\lambda_\star \sigma (t - t_n)}{\eta \varepsilon}}\right) \left\{F_{i,j}^n \mathds{1}_{v_j > 0} + F_{i+1,j}^n \mathds{1}_{v_j < 0}\right\}_{j=1, \dots, 2N}\nonumber\\
    \mbox{\textcircled{d}}&=&\sum_{k=1}^{N'}{\cal A}_k^{-1} P_k \left\{ \langle v_j e_j, \frac{\mathcal{C}(t)}{\lambda_\star} (D - \lambda_\star I) \left[-\frac{\varepsilon}{\sigma} \frac{\rho_{i+1}^n - \rho_i^n}{\Delta x}\bf 1\right] \rangle \right\}_{j=1, \dots, 2N}.  \nonumber
    \end{eqnarray}
    As previously, we now deal with the decomposition. Since 
    $\mbox{\textcircled{a}}$ and $\mbox{\textcircled{c}}$ do not depend on the gradient reconstruction, they are unchanged and respectively given by \eqref{arond} and \eqref{crond}; we  focus  on $\mbox{\textcircled{c}}$ and $\mbox{\textcircled{d}}$. 
        
    \subsubsection{Calculation of the term  {\textcircled{b}}}
    Regarding the term $\mbox{\textcircled{b}}$, we now have the following expression 
     \begin{eqnarray*}
    \mbox{\textcircled{b}} &=&  - \frac{\varepsilon {\cal C}(t)}{\sigma} \frac{1}{2N} \langle (D  - \lambda_\star I)V,  \frac{\rho_{i+1}^n - \rho_i^n}{\Delta x} \rangle {\bf 1} \nonumber\\
    &=&  - \frac{\varepsilon {\cal C}(t)}{\sigma} \frac{1}{2N} \langle V, (D  - \lambda_\star I)\frac{\rho_{i+1}^n - \rho_i^n}{\Delta x} \rangle {\bf 1}\nonumber\\
    &=&   \frac{\varepsilon \lambda_\star {\cal C}(t)}{\sigma} \frac{1}{2N} \langle V, \frac{\rho_{i+1}^n - \rho_i^n}{\Delta x} \rangle {\bf 1}=  \frac{\varepsilon \lambda_\star {\cal C}(t)}{\sigma} \frac{1}{2N} \sum_{k=1}^{2N} v_k \, \frac{\rho_{i+1}^n - \rho_i^n}{\Delta x}{\bf 1}   = {\bf 0},  
    \end{eqnarray*}
     since the velocity grid satisfies $\sum_{k=1}^{2N} v_k=0$ by assumption.

    \subsubsection{Calculation of the term {\textcircled{d}}}
    Let now investigate the term $\mbox{\textcircled{d}}$. Using the reconstruction \eqref{discrete3},  the term $\mbox{\textcircled{d}}$ becomes 
    \begin{eqnarray}
       \mbox{\textcircled{d}} &=& \sum_{k=1}^{N'} {\cal A}_k^{-1}P_k \left\{ \langle v_j e_j, {{\mathcal C}(t)} (D - \lambda_\star I) \left[-\frac{\varepsilon}{\sigma} \frac{ \rho_{i+1}^n - \rho_i^n}{\Delta x} \bigbold{1} \right] \rangle \right\}_{j=1, \dots, 2N}\nonumber\\
       &= & D^{+} \left\{ \langle v_j e_j, {{\mathcal C}(t)} (D - \lambda_\star I) \left[-\frac{\varepsilon}{\sigma} \frac{ \rho_{i+1}^n - \rho_i^n}{\Delta x} \bigbold{1} \right] \rangle \right\}_{j=1, \dots, 2N} + {\cal O}(\varepsilon^q)\nonumber\\
        &=&  {{\mathcal C}(t) \lambda_\star} D^{+} \left\{ \langle v_j e_j, \frac{\varepsilon}{\sigma} \frac{ \rho_{i+1}^n - \rho_i^n}{\Delta x} \bigbold{1} \rangle \right\}_{j=1, \dots, 2N} ={{\mathcal C}(t)\lambda_\star} D^{+} \left[\frac{\varepsilon}{\sigma} \frac{ \rho_{i+1}^n - \rho_i^n}{\Delta x}\right] V \nonumber\\
        \label{drond}
        &=& {{\mathcal C}(t)\lambda_\star}\left[\frac{\varepsilon}{\sigma} \frac{ \rho_{i+1}^n - \rho_i^n}{\Delta x}\right] U,
    \end{eqnarray}
    where we used $D \bigbold{1}=0$ and $D^+V=U$. 
We can observe that with the reconstruction \eqref{discrete3}, the knowledge of the projectors $P_k$ can be avoided but as we will see, the asymptotic behavior can still be recovered. 

     \subsubsection{Computation of the Interface value}
 All in all, inserting the expressions \eqref{arond}, \eqref{crond} and \eqref{drond} of $\mbox{\textcircled{a}},  \mbox{\textcircled{c}}$ and $\mbox{\textcircled{d}}$  (recalling that $\mbox{\textcircled{b}}=0$) in \eqref{resFi} leads to the following expression for the interface value:
    \begin{eqnarray}
    F_{i+1/2}(t) &=& \left(\rho_i^{+,n} + \rho_{i+1}^{-,n}\right) \bigbold{1} \nonumber\\ 
    &+& (I - P_0) \exp\left({\frac{\lambda_\star \sigma (t - t_n)}{\eta \varepsilon}}\right) \left\{F_{i,j}^n \mathds{1}_{v_j > 0} + F_{i+1,j}^n \mathds{1}_{v_j < 0}\right\}_{j=1, \dots, 2N} \nonumber\\
    &+& \lambda_\star \mathcal{C}(t) \left[\frac{\varepsilon}{\sigma} \frac{ \rho_{i+1}^n - \rho_i^n}{\Delta x}\right] U \nonumber\\
    &=& \exp\left({\frac{\lambda_\star \sigma (t - t_n)}{\eta \varepsilon}}\right) \left\{F_{i,j}^n \mathds{1}_{v_j > 0} + F_{i+1,j}^n \mathds{1}_{v_j < 0}\right\}_{j=1, \dots, 2N} \nonumber\\
    &&\hspace{-0.6cm} + \left(1 - \exp\left({\frac{\lambda_\star \sigma (t - t_n)}{\eta \varepsilon}}\right)\right) \left(\rho_i^{+,n} + \rho_{i+1}^{-,n}\right) \bigbold{1} + \lambda_\star \mathcal{C}(t) \left[\frac{\varepsilon}{\sigma} \frac{ \rho_{i+1}^n - \rho_i^n}{\Delta x}\right] U. \label{resFi3}
    \end{eqnarray}
    This expression of $F_{i+1/2}(t)$ will be inserted in the finite volume scheme which will lead to a new UGKS described hereafter.

    \subsection{New UGKS}
    The complete scheme consists of considering \eqref{ugks_rho}-\eqref{ugks_f} with the following definition for the flux $\phi_{i+1/2,j}$  
    $$
    \phi_{i+1/2,j}=\frac{1}{\eta \Delta t} \int_{t_n}^{t_{n+1}} v_j F_{i+1/2,j}(t) dt, 
    $$
     where the expression \eqref{resFi3} is used for the interface value $F_{i+1/2,j}(t)$. Some calculation enables to get  the following explicit expression for the flux:     
    \begin{eqnarray}
    \phi_{i+1/2,j} &=& \frac{1}{\eta \Delta t} \int_{t_n}^{t_n + \Delta t} {v_j}\Bigg[ \exp\left({\frac{\lambda_\star \sigma (t - t_n)}{\eta \varepsilon}}\right) \left(F_{i,j}^n\mathds{1}_{v_j > 0} + F_{i+1,j}^n \mathds{1}_{v_j < 0}\right)  \nonumber\\
    &&\hspace{-1cm} +  \left(1 - \exp\left({\frac{\lambda_\star \sigma (t - t_n)}{\eta \varepsilon}}\right)\right) \left(\rho_i^{+,n} + \rho_{i+1}^{-,n}\right) + \lambda_\star \mathcal{C}(t) \frac{\varepsilon}{\sigma} \frac{ \rho_{i+1}^n - \rho_i^n}{\Delta x} U_j \Bigg] dt \nonumber\\
    && \hspace{-1cm}= \mbox{A}(\Delta t, \sigma, \eta, \varepsilon) \left(F_{i,j}^n \mathds{1}_{v_j > 0} + F_{i+1,j}^n \mathds{1}_{v_j < 0}\right) v_j + \mbox{C}(\Delta t, \sigma, \eta, \varepsilon) \left(\rho_i^{+,n} + \rho_{i+1}^{-,n}\right) v_j \nonumber\\
    \label{flux_ugks}
    &&\hspace{-0.5cm} + \mbox{D}(\Delta t, \sigma, \eta, \varepsilon) \frac{ \rho_{i+1}^n - \rho_i^n}{\Delta x} \lambda_\star U_j v_j,
    \end{eqnarray}
    where the coefficients A,C,D are given by (using the notation $w= \lambda_\star\sigma\Delta t/ (\eta \varepsilon)$):
    \begin{eqnarray}
    \label{coef_A}
 \hspace{1.cm}\mbox{A}(\Delta t, \sigma, \eta, \varepsilon) &=& \frac{1}{\eta w} (e^w-1) \\ 
  \label{coef_C}
    \mbox{C}(\Delta t, \sigma, \eta, \varepsilon) &=& \frac{1}{\eta} - \mbox{A}(\Delta t, \sigma, \eta, \varepsilon) \\ 
    \mbox{D}(\Delta t, \sigma, \eta, \varepsilon) 
    \label{coef_D}
    &=& \frac{\varepsilon}{\sigma \lambda_\star} (\mbox{C}(\Delta t, \sigma, \eta, \varepsilon)-\mbox{A}(\Delta t, \sigma, \eta, \varepsilon))+\frac{\varepsilon }{\eta \sigma \lambda_\star}e^w.  
    \end{eqnarray}
    This exactly matches the scheme proposed by Mieussens in \cite{mieussens2013} for the BGK operator, if the interface value for the density used in \cite{mieussens2013} is set to $(\rho^n_i + \rho^n_{i+1})/2$.     Once the fluxes have been written, it remains to define the scheme satisfied  by $F^{n+1}_i\in \mathbb{R}^{2N}$ from \eqref{ugks_f}
    \begin{eqnarray}
    \label{ugks_new}
    \left(I - \frac{\sigma \Delta t}{\varepsilon \eta} D\right) F_i^{n+1} = F^n_i - \frac{\Delta t}{\Delta x} \left(\phi_{i+1/2} - \phi_{i-1/2}\right).
    \end{eqnarray}
    Due to the implicit treatment of the collision operator $D$, a linear system has to be inverted for each spatial cell, which is fully aligned with other strategies for stiff kinetic problems. 
    
Let us now consider the macroscopic flux $\Phi_{i+1/2}$ which is defined by 
    $$
    \Phi_{i+1/2}=\frac{1}{2N}\sum_{j=1}^{2N} \phi_{i+1/2,j},    
    $$
where the (microscopic) flux $\phi_{i+1/2,j}$ is given by  \eqref{flux_ugks}. Some simple calculations enable to get 
\begin{eqnarray*}
\Phi_{i+1/2}&=& \mbox{A}(\Delta t, \sigma, \eta, \varepsilon) (J^{+,n}_i+J^{-,n}_i) +\mbox{D}(\Delta t, \sigma, \eta, \varepsilon) \frac{\lambda_\star}{2N}\frac{ \rho_{i+1}^n - \rho_i^n}{\Delta x} \langle U, {V} \rangle,   \nonumber 
\end{eqnarray*}
with $J^{\pm,n}_i$ are given by 
$$
J_i^{-,n} =\frac{1}{2N}\sum_{j=1}^N v_j F_{i,j}^n \;\; \mbox{ and } \;\; J_i^{+,n} =\frac{1}{2N}\sum_{j=N+1}^{2N} v_j F_{i,j}^n. 
$$
But from Prop. \ref{prop1}, one has 
$\lambda_\star \langle U,  {V} \rangle = \langle V,  {V} \rangle$ so that the macroscopic flux $\Phi_{i+1/2}$ finally becomes 
\begin{eqnarray}
\label{flux_ugks_rho}
\Phi_{i+1/2}&=& \text{A}(\Delta t, \sigma, \eta, \varepsilon) (J^{+,n}_i+J^{-,n}_i)+ \frac{\text{D}(\Delta t, \sigma, \eta, \varepsilon)}{2N} \frac{ \rho_{i+1}^n - \rho_i^n}{\Delta x} \langle V, V \rangle, 
\end{eqnarray}
and the macroscopic density is updated by as 
\begin{equation}
\label{new_ugks_rho}
    \frac{\rho^{n+1}_i - \rho^n_i}{\Delta t} + \frac{\Phi_{i+1/2}-\Phi_{i-1/2}}{\Delta x} = 0.   
\end{equation}

\begin{remark}

\

\begin{itemize}
    \item  Let us remark that, contrary to the standard UGKS applied to the BGK equation, the update of $\rho^{n+1}$ is not required to update the kinetic unknown $F^{n+1}$ in \eqref{ugks_new}. Indeed, in this version of UGKS, 
 $F^{n+1}$ is updated from \eqref{ugks_new} and then, 
 $\rho^{n+1}$ is computed from \eqref{new_ugks_rho}. 
 \item Solving  $\eqref{ugks_new}$ requires a linear solver (Conjugate Gradient method in this paper) in each spatial cell if $D$ has no clear properties. When $D$ comes from BGK model, the trick is to update $\rho$  through \eqref{new_ugks_rho}: the system becomes diagonal. When $D$ is linked to FP model, one has  to use Thomas algorithm for tridiagonal matrices for more efficiency and  precision (the matrix is badly conditioned for Conjugate Gradient method). 
\end{itemize}

\end{remark}    
 
    \section{Asymptotic behavior}
 Here, we formally  investigate the asymptotic behavior of the new UGKS \eqref{ugks_new}-\eqref{flux_ugks}-\eqref{new_ugks_rho}-\eqref{flux_ugks_rho} presented above. 
 First, we study the free transport regime $\varepsilon\to +\infty$ and then the diffusion limit $\eta=\varepsilon\to 0$ is studied. 
 
 \phantomsection
 
    \subsection{Free transport regime}
 The asymptotic behavior follows from the ones of the coefficients  defined in \eqref{coef_A}, \eqref{coef_C}, \eqref{coef_D}. 
 
 \begin{proposition}
 When $\varepsilon\to +\infty$ (and $\eta$ constant), one has 
 \begin{itemize}
     \item \emph{A}$(\Delta t, \sigma, \eta, \varepsilon)$ tends to  $\frac{1}{\eta}$. 
     \item \emph{C}$(\Delta t, \sigma, \eta, \varepsilon)$ tends to $0$. 
     \item \emph{D}$(\Delta t, \sigma, \eta, \varepsilon)$ tends to $0$. 
 \end{itemize}
 \end{proposition}
As a consequence, the flux \eqref{flux_ugks} verifies when $\varepsilon\to +\infty$ 
$$
\phi_{i+1/2,j} \to \frac{v_j}{\eta} \left(F_{i,j}^n \mathds{1}_{v_j > 0} + F_{i+1,j}^n \mathds{1}_{v_j < 0}\right), 
$$
and the scheme becomes 
$$
F^{n+1}_{i,j}-F^n_{i,j} +\frac{v_j\Delta t}{\eta\Delta x} \left(F_{i,j}^n \mathds{1}_{v_j > 0} + F_{i+1,j}^n \mathds{1}_{v_j < 0}\right)  = 0, 
$$
which is the standard first order upwind scheme for the free transport equation $\partial_t f+\frac{v}{\eta}\partial_x f=0$. 
 
    \subsection{Diffusion regime}
 We now investigate the diffusion limit $\eta=\varepsilon\to 0$. 
 \begin{proposition}
 When $\eta=\varepsilon\to 0$, one has 
 \begin{itemize}
     \item \emph{A}$(\Delta t, \sigma, \eta, \varepsilon)$ tends to $0$, 
     \item \emph{D}$(\Delta t, \sigma, \eta, \varepsilon)$ tends to $\frac{1}{\sigma \lambda_\star}$. 
 \end{itemize}
 \end{proposition}
As a consequence, the macroscopic flux $\Phi_{i+1/2}$ defined by \eqref{flux_ugks_rho} satisfies,  as $\eta=\varepsilon\to 0$ 
\begin{eqnarray}
   \label{flux_rho_lim}
\Phi_{i+1/2} 
    &\to &  \frac{1}{2N}\frac{ \rho_{i+1}^n - \rho_i^n}{\Delta x} \frac{\langle V, V \rangle}{\sigma\lambda_\star}.  
\end{eqnarray}
Hence, the equation \eqref{new_ugks_rho} becomes 
$$
\frac{\rho^{n+1}_i-\rho^n_i}{\Delta t} + \frac{\langle V,V\rangle/(2N)}{\sigma \lambda_\star} \frac{\rho^n_{i+1}-\rho^n_{i}+\rho^n_{i-1}}{\Delta x^2}=0. 
$$
But from the definition \eqref{lambdastar} 
of $\lambda_\star$, we get 
$\langle V,V\rangle/\lambda_\star = \langle D^+ V,V\rangle$ which is consistent with the diffusion coefficient $\kappa$ occurring in  \eqref{eq:diffusion_limit}. 

    \section{Extensions and remarks}
    We now present some extensions and remarks.  
    
    \phantomsection
    
    \subsection{Implicit diffusion}
        
    The version where the limit scheme is implicit is obtained simply by modifying the flux as:
    \begin{eqnarray}
    \phi_{i+1/2,j} &=& \text{A}(\Delta t, \sigma, \eta, \varepsilon) \left(F_{i,j}^n \mathds{1}_{v_j > 0} + F_{i+1,j}^n \mathds{1}_{v_j < 0}\right) v_j \notag \\
    \label{flux_imp}
    &+& \text{C}(\Delta t, \sigma, \eta, \varepsilon) \left(\rho_i^{+,n} + \rho_{i+1}^{-,n}\right) v_j + \text{D}(\Delta t, \sigma, \eta, \varepsilon) \frac{\rho_{i+1}^{n+1} - \rho_{i}^{n+1}}{\Delta x} \lambda_\star U_j v_j.
    \end{eqnarray}
    However, to compute $\rho_{i}^{n+1}$, the macroscopic fluxes $\Phi_{i+1/2,j}$ are now obtained by summing on $j$ the microscopic fluxes \eqref{flux_imp}. From the previous calculations, we get 
\begin{eqnarray}
\Phi_{i+1/2} &=&    \frac{1}{2N} \sum_{j=1}^{2N} {\phi_{i+1/2,j} } \notag\\
  &=& \text{A}(\Delta t, \sigma, \eta, \varepsilon) (J_i^{+,n} + J_{i+1}^{-,n}) + \frac{\langle V, V \rangle}{2N} \text{D}(\Delta t, \sigma, \eta, \varepsilon) \frac{\rho_{i+1}^{n+1} - \rho_{i}^{n+1}}{\Delta x},
    \end{eqnarray}
    so that the scheme on $\rho$ finally becomes 
    (with the notations 
    $J_{i+1/2}^n = J_{i+1}^{-,n} + J_i^{+,n}$):
    \begin{eqnarray*}
        \frac{\rho_i^{n+1} - \rho_i^n}{\Delta t} &=& - \text{A}(\Delta t, \sigma, \eta, \varepsilon) \frac{J_{i+1/2}^n - J_{i-1/2}^n}{\Delta x} 
        - \frac{\langle V, V \rangle}{2N} \text{D}(\Delta t, \sigma, \eta, \varepsilon) \frac{\rho_{i+1}^{n+1} - 2\rho_i^{n+1} + \rho_{i-1}^{n+1}}{\Delta x^2},
    \end{eqnarray*}
    where A and D are given by \eqref{coef_A} and \eqref{coef_D}. 
    
    \begin{remark}
    Since we have an implicit diffusion in space, one needs first to solve  $\rho^{n+1}$ through a linear solver and then solve  $F^{n+1}$.
    \end{remark}

    \subsection{Some remarks for the BGK operator}
    In the BGK case, the computations simplify significantly. 
Indeed, first, all the eigenvalues $\lambda_k  (k\geq 1)$ are equal to $\lambda_\star = \textcolor{blue}{-}1$, and $\lambda_0=0$. Hence, the coefficients ${\cal A}_k$ defined in \eqref{defAk} becomes independent of $k$ and can be written as  
        $$
        {\cal A}_k=\exp\left({\frac{\lambda_\star \sigma (t - t_n)}{\eta \varepsilon}}\right) + \frac{\lambda_k}{\lambda_\star} \left(1 - \exp\left({\frac{\lambda_\star \sigma (t - t_n)}{\eta \varepsilon}}\right)\right)=1, 
        $$ 
        so that the inverse $M(t)^{-1}$ defined in \eqref{M-1} of $M(t)$ defined in \eqref{matrixM} becomes:
        \begin{eqnarray*}
        M(t)^{-1} &=& \exp\left(-{\frac{\lambda_\star \sigma (t - t_n)}{\eta \varepsilon}}\right) P_0 + \sum_{k=1}^{N'} 
        {\cal A}_k^{-1} P_k 
        = \exp\left(-{\frac{\lambda_\star \sigma (t - t_n)}{\eta \varepsilon}}\right) P_0 + (I - P_0).
        \end{eqnarray*}
     Besides, in the BGK case, one can observe that $D=P_0+\lambda_\star I$ or $D - \lambda_\star I = P_0$. It turns out that for the BGK case, the second reconstruction  \eqref{reconstruction} and the third reconstruction \eqref{third} gives  the same expression for $S(t)$ defined by \eqref{defS} and \eqref{defS2} respectively. Indeed, the reconstruction gives \eqref{reconstruction} 
        \begin{eqnarray}
        S_j(t) &=& \exp\left({\frac{\lambda_\star \sigma (t - t_n)}{\eta \varepsilon}}\right) \left(F_{i,j}^n \mathds{1}_{v_j > 0} + F_{i+1,j}^n \mathds{1}_{v_j < 0}\right) \nonumber\\
        &+& \langle v_j e_j, \int_{0}^{\frac{\sigma t}{\eta\varepsilon}} \exp\left({\lambda_\star u}\right) P_0 \left[-\frac{\varepsilon u}{\sigma} \frac{F_{i+1}^n - F_i^n}{\Delta x}\right] du \rangle \nonumber\\
        \label{S_BGK}
        &=& \exp\left({\frac{\lambda_\star \sigma (t - t_n)}{\eta \varepsilon}}\right) \left(F_{i,j}^n \mathds{1}_{v_j > 0} + F_{i+1,j}^n \mathds{1}_{v_j < 0}\right) + \frac{\mathcal{C}(t) v_j}{\lambda_\star} \frac{\varepsilon}{\sigma} \frac{ \rho_{i+1}^n - \rho_i^n}{\Delta x},
        \end{eqnarray}
        which is the expression that we would have obtained using the third reconstruction \eqref{third} (let  recall the definition \eqref{defC} of $\mathcal{C}(t)$).  
       We can observe that the second term in \eqref{S_BGK} is collinear to $V$, ie orthogonal to Ker $D$ or $P_0 V=0$. Solving the system $M(t)F_{i+1/2}(t) = S(t)$ thus leads  to:
        \begin{eqnarray*}
        F_{i+1/2}(t) &=& M(t)^{-1} S(t) \\
        &&\hspace{-1.5cm}= P_0 \left\{F_{i,j}^n \mathds{1}_{v_j > 0} + F_{i+1,j}^n \mathds{1}_{v_j < 0}\right\}_{j=1, \dots, 2N} + {\bf 0} \\
        &&\hspace{-1.5cm}+ (I - P_0) \exp\left({\frac{\lambda_\star \sigma (t - t_n)}{\eta \varepsilon}}\right) \left\{F_{i,j}^n \mathds{1}_{v_j > 0} + F_{i+1,j}^n \mathds{1}_{v_j < 0}\right\}_{j=1, \dots, 2N} + \frac{\varepsilon}{\sigma} \frac{ \rho_{i+1}^n - \rho_i^n}{\Delta x} \frac{\mathcal{C}(t) V}{\lambda_\star} \\
        &&\hspace{-1.5cm} = \left(1 - \exp\left({\frac{\lambda_\star \sigma (t - t_n)}{\eta \varepsilon}}\right)\right) \left(\rho_i^{+,n} + \rho_{{i+1}}^{-,n}\right) \bigbold{1} \\
        &&\hspace{-1.5cm}+ \exp\left({\frac{\lambda_\star \sigma (t - t_n)}{\eta \varepsilon}}\right) \left\{F_{i,j}^n \mathds{1}_{v_j > 0} + F_{i+1,j}^n \mathds{1}_{v_j < 0}\right\}_{j=1, \dots, 2N} + \frac{\varepsilon}{\sigma} \frac{ \rho_{i+1}^n - \rho_i^n}{\Delta x} \frac{\mathcal{C}(t) V}{\lambda_\star}. 
        \end{eqnarray*}
We observe that this expression and the one obtained in \cite{mieussens2013} match except for the last term 
in which half-slopes reconstructions have been used for $\rho$ in \cite{mieussens2013}. This is due to the fact that in \cite{mieussens2013}, the BGK model requires a reconstruction for $F$ and for $\rho$ which is not the case in this work. 


    \subsection{Some remarks for the Fokker-Planck operator} \label{subsecfp}
    To apply our scheme to  Fokker-Planck operator, we need to give a discretization of ${\cal D}$. Moreover, we  discuss how the interface value $F_{i+1/2}(t)$ behave in this case. 
    
    \phantomsection
    
    \subsubsection{Discretization of the collision operator}
        First, let recall that for the continuous Fokker-Planck operator defined by \eqref{eq:collision_operator_diff}, the eigenvalues are $\lambda_k=-k(k+1), k\in \mathbb{N}$. The zero eigenvalue is labelled by $k=0$ and the eigenvalue $\lambda_1$ is associated with the eigenfunction $v$ (which can easily recovered from    \eqref{eq:collision_operator_diff}). Let us remark that we will choose $\lambda_\star=\lambda_1$ since this is the eigenvalue which enables to get the correct diffusion limit. 
        
        Second, let introduce a velocity discretization  $D$ of ${\cal D}$  given by \eqref{eq:collision_operator_diff}. With the definitions introduced above, we use  {(we recall that we have chosen $c=1$ )}
        \begin{equation}
        \label{discrete_FP}
        ({\cal D}f)(v_j)\approx (DF)_j = \frac{1}{\Delta v^2}\Big((1-v_{j+1/2}^2) (F_{j+1}-F_j) - (1-v_{j-1/2}^2) (F_{j}-F_{j-1})\Big), 
        \end{equation}
        where we defined $v_{j+1/2}=(v_j+v_{j+1})/2$. 
        Using this approximation, one can check that 
        the first two eigenvalues $\lambda_0=0$ and $\lambda_1=-2$ are exactly recovered.
        \begin{proposition}
        \label{prop2}
        The discretization \eqref{discrete_FP} satisfies  $D{\bf 1}=0$ and $DV=-2V$. Besides $D$ is symmetric, its non diagonal coefficients are positive and $I+\delta D$ is  bistochastic  for $\delta>0$ small enough.
        \end{proposition}
        \begin{proof}   
         The fact that $\lambda_0=0$ is true is a direct consequence of the conservative form of the discretization \eqref{discrete_FP} 
        so that $D{\bf 1}=0$. 
        
        Regarding the relation, let replace $F_j$ by $v_j$ in \eqref{discrete_FP}: 
        \begin{eqnarray*}
        (D V)_j &=& \frac{1}{\Delta v^2}\Big((1-v_{j+1/2}^2) (v_{j+1}-v_j) - (1-v_{j-1/2}^2) (v_{j}-v_{j-1})\Big) \nonumber\\
        &=& \frac{1}{\Delta v^2}\Big((1-v_{j+1/2}^2) \Delta v - (1-v_{j-1/2}^2) \Delta v\Big) \nonumber\\
        &=& \frac{1}{\Delta v}(-v_{j+1/2}^2+v_{j-1/2}^2)  = -2 v_j, 
        \end{eqnarray*}
        so that $DV=-2V$ and $V$ is an  eigenvector associated to the eigenvalue $\lambda_\star=-2$.

        By construction, we clearly have that $D$ is symmetric, its non diagonal coefficients are positive and $I+\delta D$ is a bistochastic matrix for $\delta>0$ small enough.
                \end{proof}

  
  \subsubsection{Interface value}
In the Fokker-Planck case, the eigenstructure of $D$ is more involved than in the BGK case so that the third reconstruction \eqref{third} is required to avoid the knowledge of the projectors $P_k$. However, it can be noticed that for the second reconstruction \eqref{reconstruction}, the term $\mbox{\textcircled{b}}$ defined in \eqref{termb_1} vanishes. Indeed, thanks to Prop. \ref{prop2}, one has $(D-\lambda_\star I)V=0$, even if $F^n_{i,j}$ is employed for the gradient. 
    
    \subsection{Some remarks for the scattering operator} \label{subsecscat}
   
    For scattering operator defined by \eqref{eq:collision_operator_scat},
   the resulting matrix $D$ 
    must be symmetric negative with only $\bigbold{1}$ in the kernel and  positive coefficients apart on the diagonal. 
   Spectrum and projectors are not always exactly known, thus $\lambda_\star$ has to be numerically determined from  \eqref{lambdastar} and in this case  reconstruction \eqref{discrete3} is fully justified. 
    Up to our knowledge, there is no physically relevant non isotropic  1D collision   operator. Therefore, to test our approach, we  artificially define the following scattering matrix $D=\frac1{10}D_1$ with $D_1$  the periodic Laplacian matrix given by :
 $$
  {(D_1)}_{i,j} = \frac{1}{\Delta v^2}
  \begin{cases}
  \displaystyle -2, & \text{if } i = j, \\   \displaystyle 1 , & \text{if } i = j-1 [N_v] \mbox{ or }  i = j+1 [N_v] \quad (N_v=2N)
  \end{cases}
  $$
 which owns the required properties (negativity, kernel, positive coefficients).
   The factor $1/10$ is chosen so that the $\lambda_\star$ value in this case is between the one of BGK model and the one FP model. The eigenvalues of $D_1$  are known: \[
\lambda_k = -\frac{2}{\Delta v^2} \left(1 - \cos\left( \frac{2\pi k}{N_v} \right) \right), \quad k = 0, 1, \dots, N_v-1. 
\]

 \noindent  The numerical value for $\lambda_\star(D)$ is obtained by the code for $N_v=100$ is roughly $-1.49835$.
   \begin{remark} \
  \noindent    The vector $V$ is not an eigenvector since eigenvectors are discrete Fourier modes. Using Fourier analysis one gets that  when $\Delta v$ tends to zero:
 \begin{itemize}
 \item  {$D_1$ tends to the Laplacian so that $U$ obeys to $U''(v)=v \; (v\in [-1, 1]$) with periodic boundary conditions $U(1)=U(-1)$, $U$ and satisfies $\frac 12\int_{-1}^1 U(v) dv=0$  (to be in the orthogonal of  the constants). The solution $U$ is $U(v)=\frac16 v^3-\frac16v$. }
 
  \item One can compute the corresponding  pseudo-eigenvalue $\lambda_\star$ by considering the limit as $\Delta v\to 0$ in \eqref{lambdastar}. Using $\int_{-1}^1 U(v) v dv=-1/45$ and $\int_{-1}^1 v^2 dv=1/3$, one has 
 $$\lambda_\star(D_1): = \frac{\langle V, V \rangle}{\langle U, V \rangle} \underset{\Delta v \rightarrow 0}{\longrightarrow}\left.\left(\frac 12\int_{-1}^1 v^2dv\right)\right/\left(\frac 12\int_{-1}^1 U(v){v} dv\right)=-15. 
 $$
Since $D=\frac1{10}D_1$,  $\lambda_\star(D)$ converges to $\lambda_\star^{\infty}(D)=-1.5$ in the continuous limit.
 \end{itemize}
   \end{remark}

 \

    \section{Numerical results}

    We now provide  some numerical results to illustrate the properties of the new UGKS to solve \eqref{eq:boltzmann} for generalized collision kernels, namely the scattering operator  \eqref{eq:collision_operator_scat} and the Fokker-Planck operator  \eqref{eq:collision_operator_diff}. 
In all  cases  presented below, the equation \eqref{eq:boltzmann} is equipped with  periodic conditions in space (more details on other boundary conditions can be found in \cite{mieussens2013} and the new scheme only needs slight adaptations on the boundaries). The initial particle distribution function and density are chosen far from equilibrium:
\begin{eqnarray*}
f_0(x, v) &= &\exp\left( - (x - 0.5)^2 - 10(1 - v)^2 \right), \quad 0 \leq x \leq 1,\ -1 \leq v \leq 1. \\
\rho_0(x)&=&\frac12\int_{-1}^{1} f_0(x, v) \, dv = C\exp\left( - (x - 0.5)^2 \right), \\
\mbox{with  }\quad C&=&\frac12\int_{-1}^{1} \exp\left( -10 (1 - v)^2 \right) dv \approx 0.14  \qquad \mbox{ (see figure } \eqref{init_rho}).
\end{eqnarray*} 

\begin{figure}[H]
\label{init_rho}
  \centering
    \includegraphics[width=0.55\linewidth,  trim=2.8cm 1.38cm 3.6cm 1.28cm,  clip] {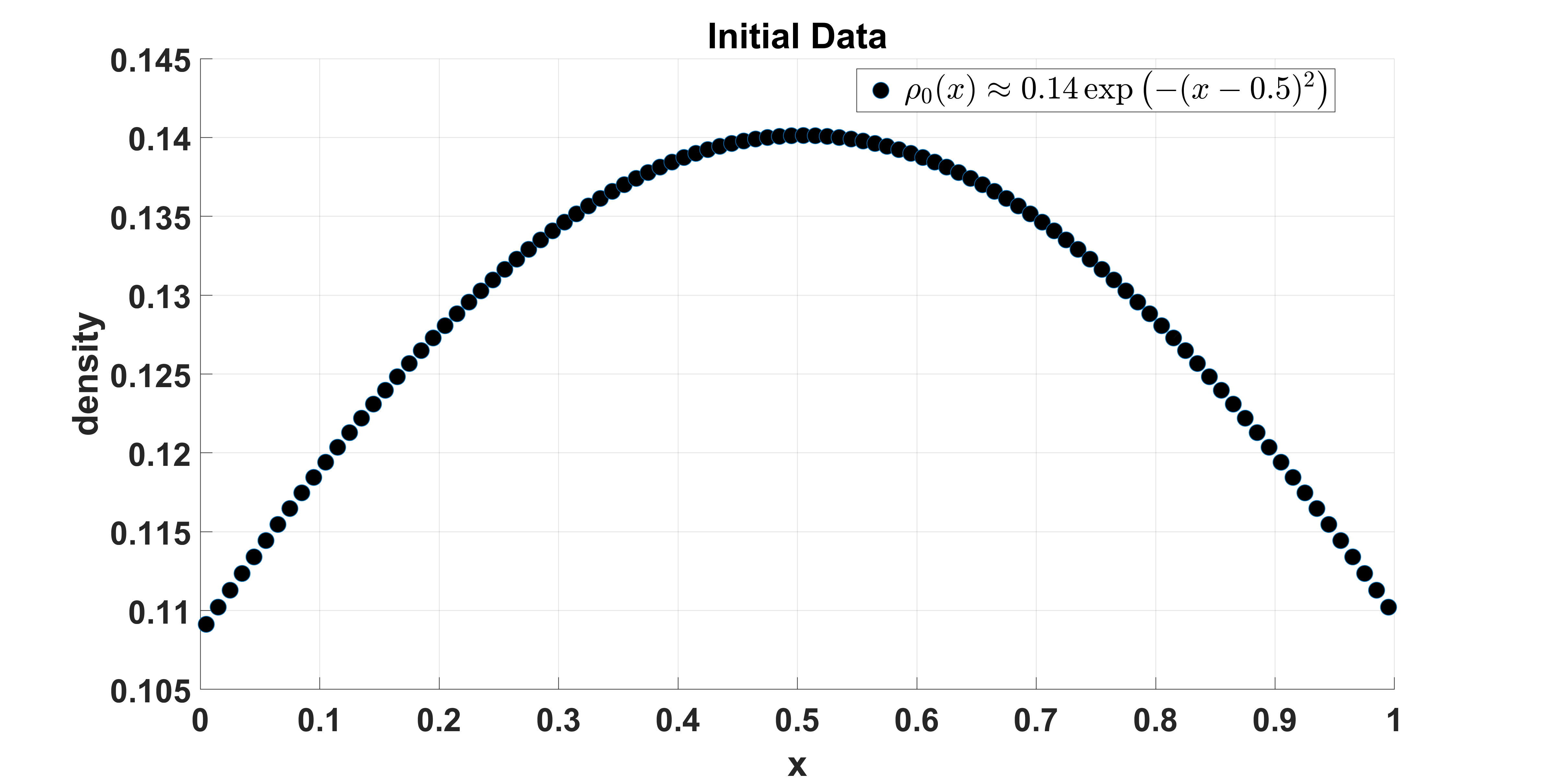}
    \caption{Initial density $x\mapsto \rho_0(x)$}
    \end{figure}
   The exact solution of the transport problem without any collision is
   \begin{equation} \label{reftransp}
       f(t, x, v) = f_0\big( (x - v t) \bmod 1, \, v \big) = \exp\left(- \big( (x - v t) \bmod 1 - 0.5 \big)^2 - 10 (1 - v)^2 \right), 
   \end{equation}
    while the exact solution for density in the diffusive limit with diffusion coefficient $\ds {\kappa} = \frac{1}{3\sigma \abs{\lambda_\star} }$ is:
    \begin{equation} \label{refdiff}
\rho(t,x) = C \int_0^1 \left( \sum_{j = -\infty}^{\infty} \frac{1}{\sqrt{4\pi{\kappa} t}} \exp\left( -\frac{(x - y + j)^2}{4 {\kappa} t} \right) \right) \exp\left( - (y - 0.5)^2 \right) dy.
 \end{equation}
 .

In the sequel, we  focus on the comparison of different collision operators and their corresponding  diffusion coefficients $\kappa$: the scattering operator  \eqref{eq:collision_operator_scat} for which $\kappa=\frac{2}{9\sigma}$ (since $\lambda_\star=-3/2$), the Fokker-Planck operator  \eqref{eq:collision_operator_diff}  for which $\kappa=\frac{1}{6\sigma}$ (since $\lambda_\star=-2$) and BGK operator for which $\kappa=\frac{1}{3\sigma}$ (since $\lambda_\star=-1$).
Regarding the stability, it is empirically found that the new UGKS has the same stability condition as the standard one studied in \cite{mieussens2013, vigier}.  
The time step is thus chosen 
so that the natural condition $\Delta t=C_1\Delta x^2+C_2\eta\Delta x$ is empirically satisfied. We refer to the work of Vigier (\cite{vigier}) for more discussions on the subject.
     
In the following tests, the mesh in space uses $N_x=100$ cells whereas $N_v=100$ cells in velocity ($N=50$) are considered. 
The quantity $\sigma$ will always be set to one. Hence, from the empirical CFL condition, the time step is set to $\Delta t=10^{-5}$. 

The next subsections investigate the  three different regimes (transport, intermediate and diffusive) for the three models at different times. 
 In the legends BGK means BGK model, FP means Fokker-Planck model and SC refers to the scattering model.

    \subsection{Transport regime: $\eta=1$ and $\varepsilon=100$}
    First, we consider the so-called transport regime 
    in \eqref{eq:boltzmann} with $\eta=1$ and $\varepsilon=100$. In this regime, the collision part is weakened due to the large value of $\varepsilon$ and we capture the transport of the density: initially the mean velocity of the distribution is uniform  and strictly positive so that a translation of the bump to the right is expected. In Figure \ref{fig1}, we can observe that due to the first order scheme in transport, diffusion in space  and smoothing of the bump occur.
    In this regime, there is almost no collisions so that consequently the three kernels of collisions  give the same results at intermediate time $t_i=0.05$ and final time $t_f=0.1$.

    \begin{figure}[H]
  \centering
    \includegraphics[width=0.9\linewidth,  trim=2cm 0.25cm 3.75cm 0.5cm,  clip] {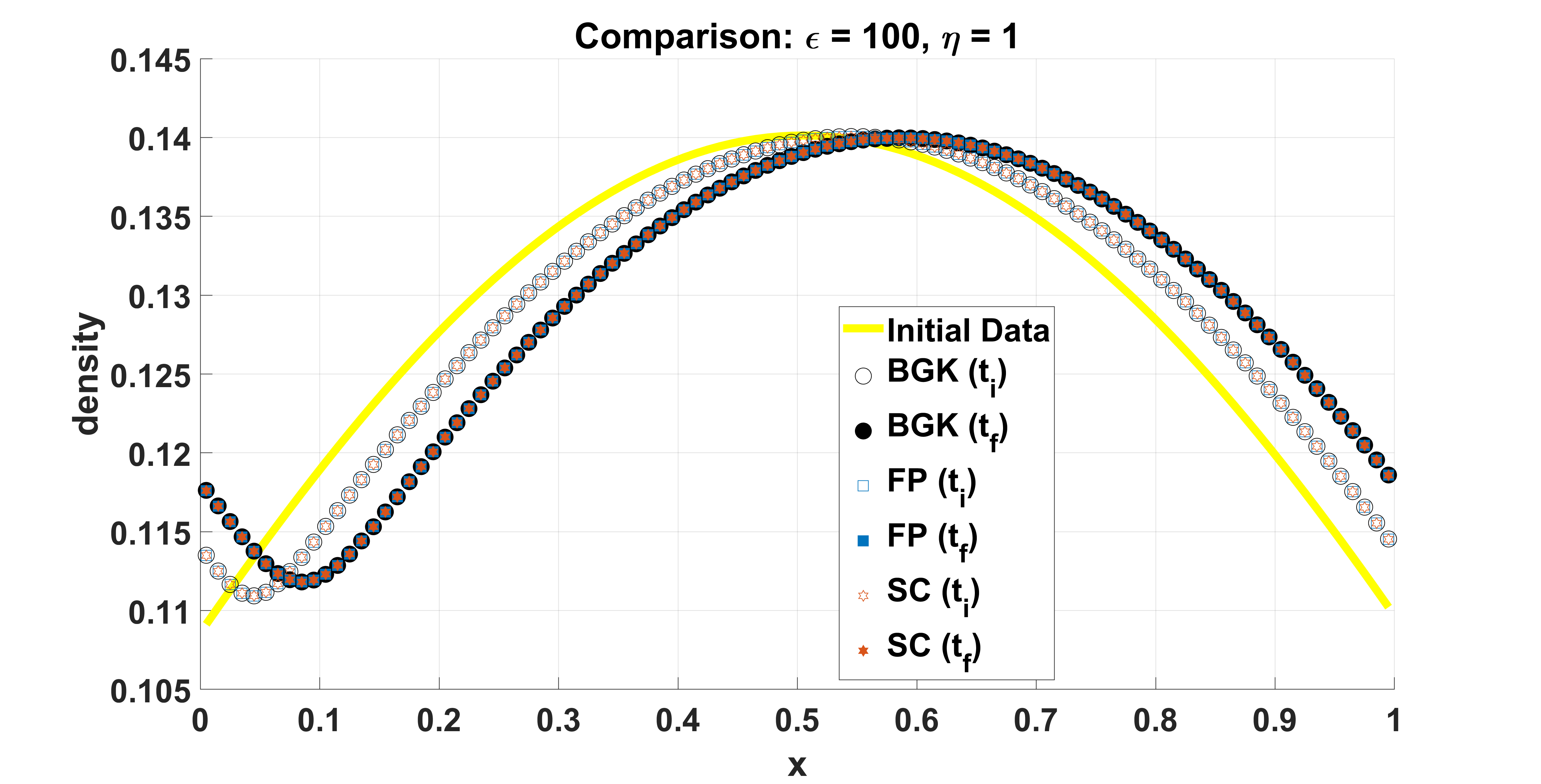}
    \caption{Transport regime ($\varepsilon=100, \eta=1$): density $x\mapsto \rho(t,x)$ for $t_i=0.05, t_f=0.1$ computed by UGKS for BGK, Fokker-Planck and scattering collision kernels. }
    \label{fig1}
    \end{figure}
    
    We can also compare with exact solution of the transport solution given  by \eqref{reftransp} and  denoted by $\mbox{Transp}$ at intermediate time and final time: we can see in Figure \ref{fig2} that the three models have exactly the same behavior
    and that the error comes from the first order upwind scheme.
       \begin{figure}[H]
  \centering
    \includegraphics[width=0.75\linewidth,  trim=2cm 0.25cm 3.cm 0.5cm,  clip] {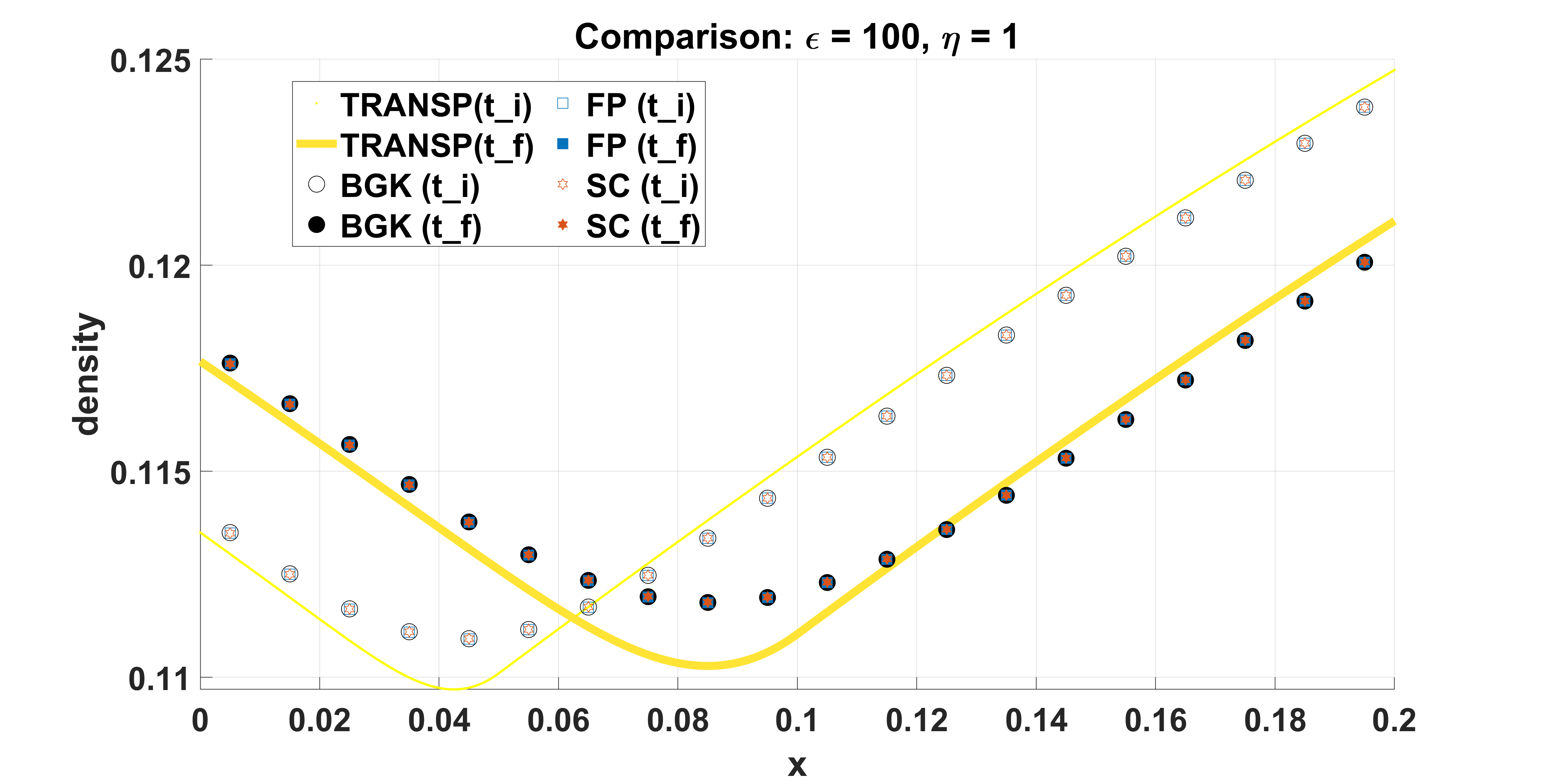}
    \caption{Transport regime ($\varepsilon=100, \eta=1$): density $x\mapsto \rho(t,x)$ for $t_i=0.05, t_f=0.1$ computed by UGKS for BGK, Fokker-Planck and scattering collision kernels compared with the exact solution. }
      \label{fig2}
    \end{figure}

    \subsection{Intermediate regime: $\eta=0.1$ and $\varepsilon=0.1$}
Now, we investigate the intermediate regime for which $\eta=\varepsilon=0.1$ in \eqref{eq:boltzmann}. In this regime both transport and collisions  are acting: transport is initially acting (the bump has stopped moving at time $t_i=0.05$) and afterwards the bump is only damped till $t_f=0.1$. Consequently, 
due to the transport part which is the same for all the models, the bump goes to the right but due to the different diffusion coefficient, the damping is different according to the collision operator:
the diffusion coefficient of the Fokker-Planck model is half the one of BGK model and three quarters of the scattering model. Thus, one can observe on Figure  \ref{figinterm} that the BGK solution (black bullets) is more flattened than the one obtained by the scattering operator (orange stars) which is itself more flattened than the Fokker-Planck operator (blue squares) at time $t_f=0.1$. 
    \begin{figure}[H]
  \centering
         \includegraphics[width=0.9\linewidth,  trim=2cm 0.25cm 3.75cm 0.5cm,  clip]{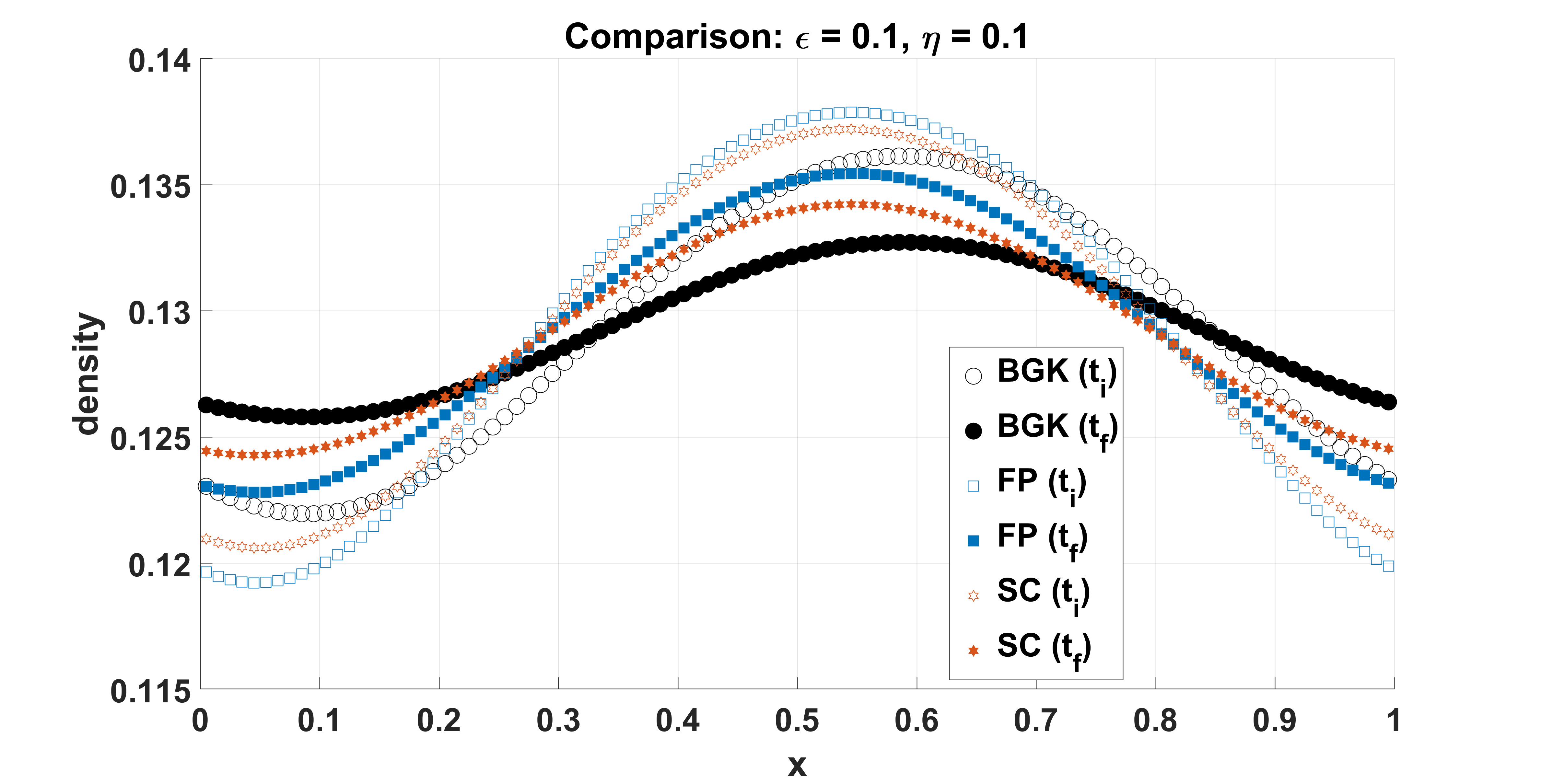}
     \caption{Intermediate regime ($\varepsilon=\eta=0.1$): density $x\mapsto \rho(t,x)$ for $t_i=0.05, t_f=0.1$ computed by UGKS for BGK, Fokker-Planck and scattering collision kernels.}
      \label{figinterm}
   \end{figure}
   
    \subsection{Diffusive regime 
    $\eta=10^{-4}$ and $\varepsilon=10^{-4}$} 
    We finally consider the diffusion regime $\eta=\varepsilon=10^{-4}$. In this regime, the effect of the transport is negligible and 
    the solution is supposed to be closed to the one of the diffusion model \eqref{eq:diffusion_limit} for which the diffusion coefficient is different according to the collision operator. As a consequence, since the diffusion coefficient in the Fokker-Planck model is half the one of the BGK model, the Fokker-Planck model at time $t_f$ should be very close to the BGK model at time $t_i=t_f/2$. 
    This is observed in Figure \ref{figdiff} where the Fokker-Planck density (blue squares) is superimposed with the BGK one (black circles). Moreover, in Figure \ref{figdiff}, the densities obtained by the three collision operators at $t_i$ and $t_f$ are gradually damped from Fokker-Planck at $t_i$ to BGK at $t_f$. 

    \begin{figure}[H]
  \centering
    \includegraphics[width=0.9\linewidth,  trim=2cm 0.25cm 3.75cm 0.5cm,  clip]{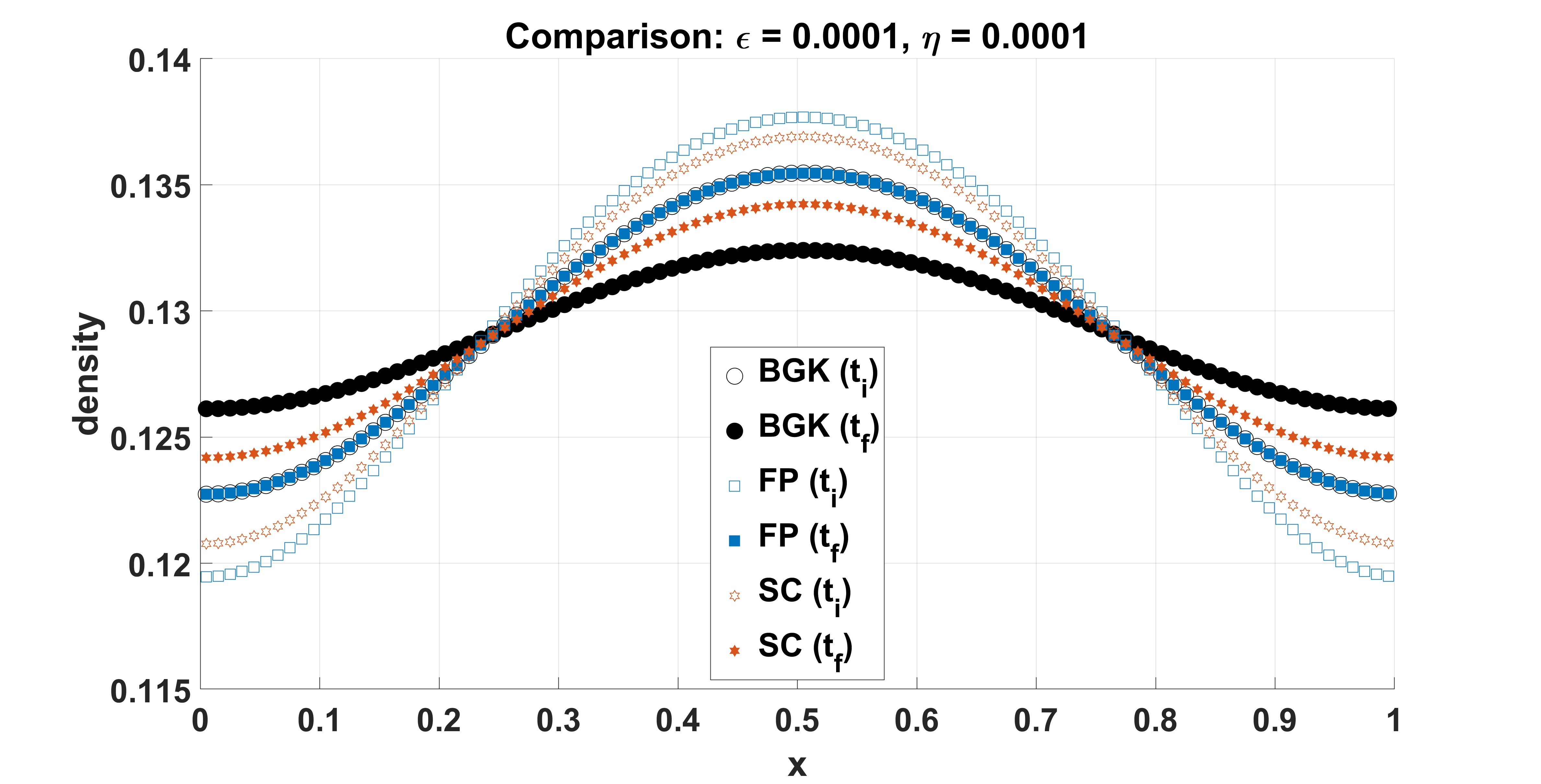}
    \caption{Diffusive regime ($\varepsilon=\eta=10^{-4}$): density $x\mapsto \rho(t,x)$ for $t=0.05$ and $t=0.1$ computed by UGKS for BGK, Fokker-Planck and scattering collision kernels.}
    \label{figdiff}
     \end{figure}
     
     Similarly, if we set $t_f=0.075$, the scattering model (orange stars) at time $t_f$ gives very similar result as the BGK model (black circles) at time $t_i=0.05$ as observed in Figure \ref{figdiff2}. Indeed, the diffusion coefficient for the scattering operator is $\lambda_{\star, SC}=-1.5$ whereas $\lambda_{\star, BGK}=-1$ for the BGK, so that       $\frac{t_f}{\lambda_{\star, SC}}=\frac{t_i}{\lambda_{\star, BGK}}$.  
         \begin{figure}[H]
  \centering
    \includegraphics[width=0.9\linewidth,  trim=2cm 0.25cm 3.75cm 0.5cm,  clip]{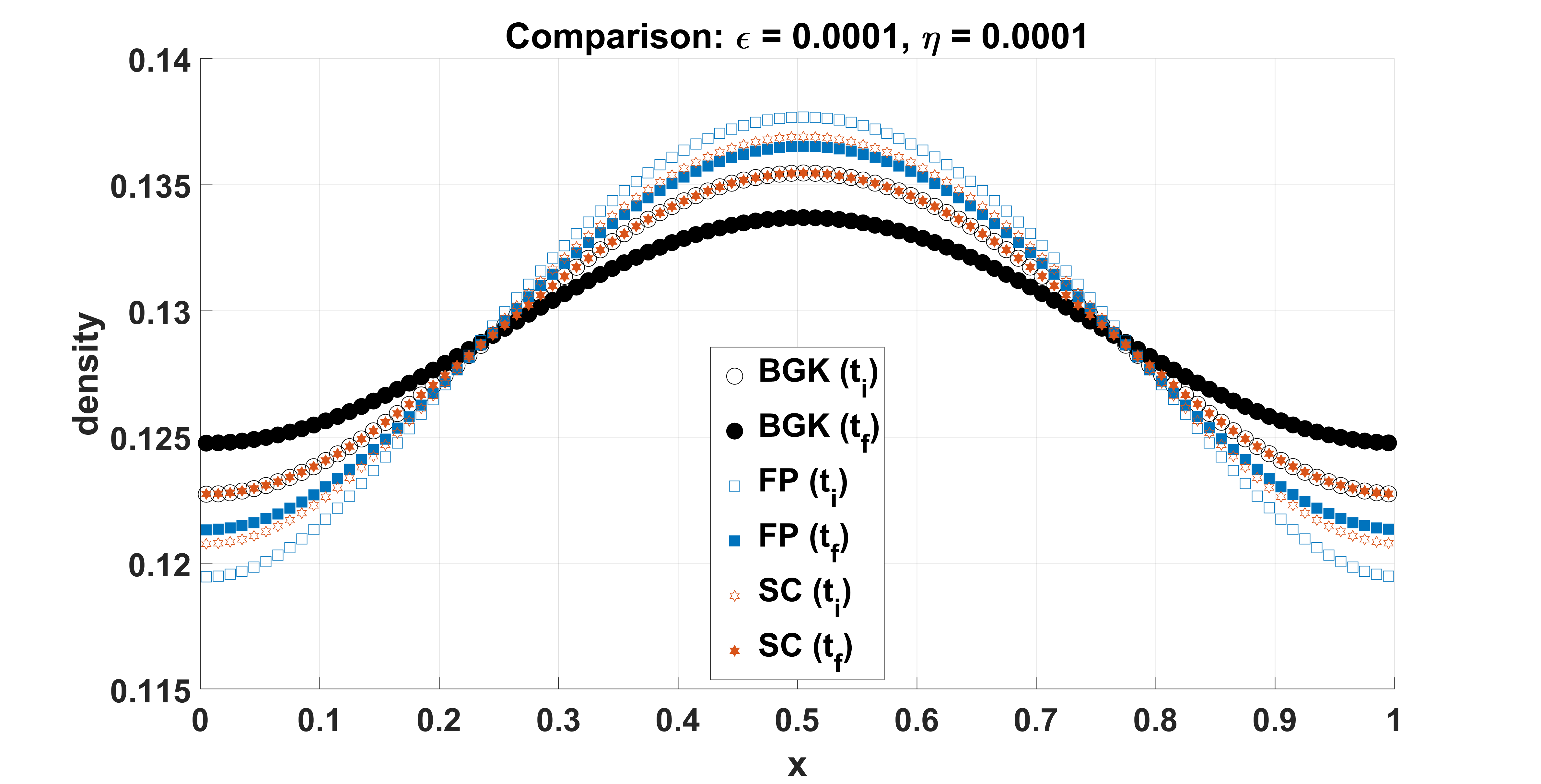}
    \caption{Diffusive regime ($\varepsilon=\eta=10^{-4}$): density $x\mapsto \rho(t,x)$ for $t=0.05$ and $t=0.075$ computed by UGKS for BGK, Fokker-Planck and scattering collision kernels.}
    \label{figdiff2}
     \end{figure}
     
     \subsubsection{Comparison with exact solution}

  We can also compare the exact solutions of the Fokker-Planck and scattering operators with their respective analytical solutions \eqref{refdiff}, denoted by $DIFF_{FP}$ and $DIFF_{SC}$, at the final time {$t=0.1$ }. In Figure~\ref{figcompdiff}, one can observe that the results are quite satisfactory after $10^4$ time steps, given the small number of velocity points and the use of a first-order time scheme. 

  \begin{figure}[H] \label{figcompdiff}
  \centering
    \includegraphics[width=0.9\linewidth,  trim=2cm 0.25cm 3.75cm 0.5cm,  clip]{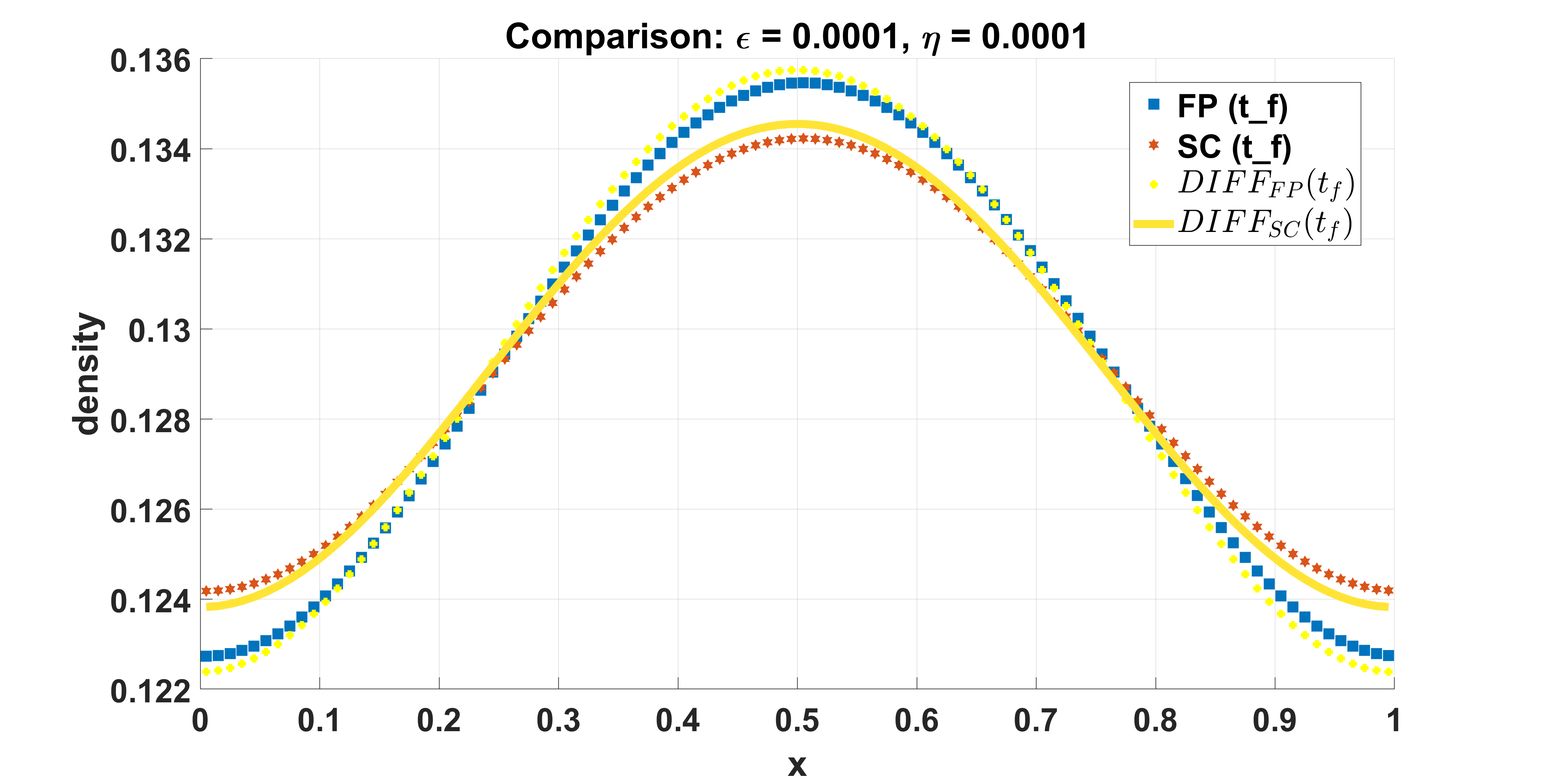}
    \caption{Diffusive regime ($\varepsilon=\eta=10^{-4}$): density $x\mapsto \rho(t,x)$ for $t=0.1$ computed by UGKS for Fokker-Planck and scattering collision kernels compared with their exact solutions.}
    
     \end{figure}     
     
    \section*{Conclusion}
 In this work, we have developed and analysed a generalized Unified Gas Kinetic Scheme (UGKS) designed to handle linear kinetic equations in the diffusive limit, with a particular focus on extending the original framework to accommodate a broad class of collision operators, such as the BGK, Fokker–Planck, and scattering models. Building on the asymptotic preserving (AP) structure of the classical UGKS, our formulation is based on a penalized Duhamel representation and a suitable spatial reconstruction  that preserve the essential asymptotic behaviors without relying on costly spectral decomposition.

From the theoretical standpoint, we formally established the preservation of the diffusion limit as $\varepsilon = \eta \to 0$. In particular, we showed that the scheme yields the appropriate diffusion equation  for each collision operator.
The practical effectiveness of the method was demonstrated through several numerical experiments covering a range of physical regimes. 

 The implicit treatment of the stiff collision term leads to local linear systems that are well-posed, and can be solved using standard solvers. The proposed interface value computation, essential for flux evaluations, preserves both accuracy and efficiency.

These results collectively demonstrate that the generalized UGKS scheme provides a  versatile numerical framework for simulating general kinetic models across multiple scales. As a future work, we may explore the extension of the scheme to nonlinear problems and rigorous stability analysis.

    \bibliographystyle{siamplain}
    \bibliography{references}

    \newpage

\end{document}